\documentclass[amsfonts]{amsart}

\usepackage{mathrsfs}
\usepackage{amsmath} 
\usepackage{ amssymb }
\usepackage{amsthm}
\usepackage{url}
\usepackage{graphicx}		% to display figures 
\usepackage{enumitem}
\usepackage{comment}    %added March 13
\usepackage[all]{xy}
\usepackage{color}
\usepackage{bm}

\newcommand{\bitem}{\begin{itemize}[leftmargin=30pt]}
\newcommand{\benum}{\begin{enumerate}[leftmargin=30pt, 
label=\rm{(\alph*)}]}
\newcommand{\benumm}{\begin{enumerate}[leftmargin=*, 
label=\rm{(\alph*)}]}

\newcommand{\R}{\mathbb{R}}
\newcommand{\C}{\mathbb{C}} 
\newcommand{\Z}{\mathbb{Z}}

\newcommand{\Inn}{\mathrm{Inn}}
\newcommand{\Out}{\mathrm{Out}}
\newcommand{\SL}{\mathrm{SL}}
\newcommand{\SO}{\mathrm{SO}}
\newcommand{\SU}{\mathrm{SU}}
\newcommand{\GL}{\mathrm{GL}}
\newcommand{\PGL}{\mathrm{PGL}}
\newcommand{\PSL}{\mathrm{ PSL}}
\newcommand{\U}{\mathrm{U}}
\newcommand{\PU}{\mathrm{PU}}

\newcommand{\diag}{\mathrm{diag}}
\newcommand{\Aut}{{\rm Aut}}
\newcommand{\ad}{{\rm ad}}
\newcommand{\E}{{\rm E}_2}
\newcommand{\CE}{\rm{CE}}

\renewcommand{\sl}{\mathfrak{sl}} 
\newcommand{\h}{{\mathfrak h}} 
\newcommand{\su}{\mathfrak{su}}
\newcommand{\aut}{\mathfrak{aut}}
\newcommand{\g}{\mathfrak{g}}
\newcommand{\sut}{\su_2}
\newcommand{\suto}{\su_{2,1}}

\newcommand{\SUt}{\SU_2}
\newcommand{\SOt}{\SO_3}
\newcommand{\PUto}{\PU_{2,1}}
\newcommand{\Pgc}{\P(\g_\C)}
\newcommand{\Pgcr}{\P(\g_\C)_{\mathrm{reg}}}
\newcommand{\Psltc}{\P(\sltc)}
\newcommand{\Psltcr}{\P(\sltc)_{\mathrm{reg}}}

\newcommand{\GLt}{\GL_2(\R)}
\newcommand{\SLt}{{\SL_2(\R)}} 
 
\newcommand{\PGLt}{\PGL_2(\R)}
\newcommand{\SUto}{\SU_{2,1}}

\newcommand{\sltc}{\sl_2(\C)}
\newcommand{\slt}{\sl_2(\R)}

\newcommand{\SLtc}{\SL_2(\C)}
\newcommand{\PSLt}{\PSL_2(\R)} 
\newcommand{\CP}{\C\P}
\newcommand{\RP}{\R\P}
\newcommand{\RPt}{\RP^2}

\newcommand{\VS}{V}
\newcommand{\DS}{D}

\newcommand{\Ee}{\mathfrak{e}}% Gil had renew here

\newcommand{\e}{{\mathbf e}}

\newcommand{\CR}{\mathrm{CR}}

\newcommand{\iso}{\stackrel{\sim}{\to}}

\newcommand{\Span}{\mathrm{Span}}

\newcommand{\II}{\mathrm I}

\renewcommand{\>}{\rangle}
\newcommand{\<}{\langle} 
\newcommand{\w}{ {\mathbf w} }

\newcommand{\X}{\mathfrak{R}}

\newcommand{\tr}{\mathrm{tr}}

\renewcommand{\P}{\mathrm{P}}

\renewcommand{\Im}{{\rm Im}}
\renewcommand{\Re}{{\rm Re}}

 \newcommand{\CC}{\mathcal{C}}

\newcommand{\st}{\, | \,}

\newcommand{\End}{{\rm End}} 
\newcommand{\Ker}{{\rm Ker}}
\newcommand{\Ad}{{\rm Ad}} 
\renewcommand{\mod}{{\rm mod\; }}

\newcommand{\n}{\noindent}

\newcommand{\sn}{\smallskip\n} 

\newcommand{\mn}{\medskip\noindent}

\newcommand{\be}{\begin{equation}}
\newcommand{\ee}{\end{equation}}

\renewcommand{\d}{\mathrm{d}}

\newcommand{\vb}{{\bf v}}

 \newcommand{\OO}{\mathcal O}

\newcommand{\wa}{well-adapted }
\newcommand{\sign}{\mathrm{sign}}

\newtheorem{thm}{Theorem}[section]
\newtheorem{lemma}{Lemma}[section]
\newtheorem{prop}{Proposition}[section]
\newtheorem{cor}{Corollary}[section]

\newtheorem{defn}{Definition}[section]

\theoremstyle{remark}

\newtheorem{rmrk}{Remark}[section]

\usepackage[margin=1.5in]{geometry}

\usepackage{xcolor}
\newcommand{\hl}[1]{#1}

\title {Left-invariant CR structures on 3-dimensional Lie groups}
\author{Gil Bor}
\address{
Centro de Intevetigaci\'on en Matem\'aticas (CIMAT), 
Guanajuato, Mexico}
\email{gil@cimat.mx} 
\author{Howard Jacobowitz}
\address{
Department of Mathematical Sciences, Rutgers University, Camden,
New Jersey, USA}
\email{jacobowi@rutgers.edu}

\date{\today}

\begin{document}
\maketitle

\begin{abstract}
The systematic study of CR manifolds originated in two pioneering 1932 papers of \'Elie Cartan. In the first, Cartan classifies all homogeneous CR 3-manifolds, the most well-known case of which is a  one-parameter family of left-invariant CR structures on $\SUt = S^3$, deforming the standard `spherical' structure. In this   paper, mostly expository,  we illustrate and clarify  Cartan's results and methods by providing
detailed classification results in modern language for four 3-dimensional Lie groups. \hl{In particular, we find that   $\SLt$}
admits two one-parameter families of left-invariant CR structures, called the  elliptic and hyperbolic families, characterized by the 
incidence of the   contact distribution with the null cone of the Killing metric. Low dimensional complex representations of $\SLt$ provide CR embedding or immersions of these structures. The same methods apply to all other three-dimensional Lie groups and are
illustrated by descriptions of the left-invariant CR structures for $\SUt$, the Heisenberg group,  and the Euclidean group.
 \end{abstract}

\tableofcontents
\section{Introduction}
A real hypersurface $M^3$ in a 2-dimensional complex manifold (such as $\C^2$) inherits an intrinsic geometric structure from the complex structure of its ambient space. This is called a CR structure and can be thought of as an odd-dimensional version of a complex structure.  A more precise definition is given \hl{in \S2} below.  The study of these  structures is based on three foundational papers.  \hl{The first is a 1907 paper of H.~Poincar\'e \cite{Po}, }which  shows that the Riemann Mapping Theorem for domains in $\C^1$ does not hold in higher dimensions.  In fact, it fails even locally, even in the real analytic case, and for the simplest of reasons:  There are more germs of real hypersurfaces than germs of holomorphic mappings.  More explicitly, Poincar\'e's observation was that for $n\geq 2$ and $N$ large enough, the space of $N$-jets of biholomorphic mappings on open sets of $\C^n$ is of lower dimension than the space of $N$-jets of real-valued functions of $2n-1$ real variables.  From this it follows that a generic perturbation of a smoothly bounded open set $A\subset {\C^n}$ is not biholomorphically equivalent to $A$ and so the Riemann Mapping Theorem fails.  It also follows that, unlike complex structures, CR structures possess {\em local} invariants, similar to the well-known curvature invariants of Riemannian metrics. Consequently, a generic CR manifold   admits {\em no} CR symmetries, even locally. 

The second foundational paper, published in two parts, is \'Elie Cartan's work of 1932 \cite {Ca1},\cite{Ca2}.  Since there are these local, indeed pointwise, invariants it is natural to  find them explicitly.  This fit in nicely with research Cartan was already doing.  The Erlangen Program of F. Klein emphasized that geometry was the study of the invariance properties of groups of transformations.  Cartan had taken this one step further with his theory of moving frames focusing on the infinitesimal action of the transformation groups.  This not only incorporated Riemannian manifolds and its generalizations into the Erlangen scheme but also provided Cartan with the new tools to study projective geometries, both real and complex, the conformal and projective deformations of surfaces, etc.  A contemporaneous explanation of Cartan's moving frames approach may be found in Weyl's review of one of Cartan's books \cite {Weyl}.  A more accessible explanation, using modern notation, is the influential article \cite{Gr} and, more recently, the graduate textbook \cite{Cl}.

Two highly significant papers that have continued and extended this study of the geometric properties of CR structures on hypersurfaces in $\C^n$ are \cite {CM} and \cite{We}

The third foundational paper took the study of CR structures in a new and surprising direction.  This was Hans Lewy's discovery, in 1957, of a locally non-solvable linear partial differential equation \cite{Le}.  To emphasize how surprising this discovery was, we quote Treves \cite{Tr}, one of the originators of the modern theory of \hl{PDEs}:

\begin{quotation}
Allow me to insert a personal anecdote: in 1955 I was given the following thesis problem: prove that every linear partial differential equation with smooth coefficients, not vanishing identically at some point, is locally solvable at that point. My thesis director was, and still is, a leading analyst; his suggestion simply shows that, at that time, nobody had any inkling of the structure underlying the local solvability problem, as it is now gradually revealed.
\end{quotation}

%'s example is a generator for  simplest (non-trivial) CR structure.  
Lewy's example is that 
for a generic \hl{smooth} $f(x,y,u)$ the equation 
\[
\left(\frac {\partial}{\partial x} + i\frac {\partial}{\partial y}-i(x+iy)
\frac {\partial }{\partial u}\right)w =f
\]
has no solution in any neighborhood of any point in $\R^3$.  

The connection of Lewy's paper to CR structures is this:  The operator on the left is  induced by the Cauchy-Riemann equations on $\C^2$ and defines the CR structure on the hyperquadric
\[
\Im (w)= |{z}|^2.
\]

This connection between CR structures and the theory of \hl{PDEs} has led to a vast amount of research such as \cite{Ho} and the three subsequent volumes for  general solvability theory and \cite{TrBook} for the study of the induced CR complex and its generalizations.  

 Here we come to the origin of the name of this field.  Cartan called these structures pseudoconformal, emphasizing that they should be thought of as a generalization of the conformal (\hl{i.e.} complex) structure of $\R^2$.  With the realization that the partial differential operators of $M^{2n+1}\subset \C^{n+1}$ induced by the Cauchy-Riemannn equations were of fundamental importance, the induced structure became known as a CR structure.  This new name was introduced by Greenfield \cite{Gr}.
 
 It is interesting to note that H. Lewy once commented (to \hl{the 2nd author}) that he was led to his example while trying to understand Cartan's paper \hl{\cite{Ca1}}.
 
 In the present article we study the CR structures most closely related to the Erlangen Program, namely the left-invariant CR structures on three-dimensional Lie groups and, more generally, homogeneous three-dimensional CR structures.  In fact,  before using the moving frames method to study the general case, Cartan used a more algebraic approach to classify  in Chapter II of  \cite{Ca1}  {\em homogeneous} CR 3-manifolds, \hl{i.e.} 3-dimensional CR manifolds admitting a transitive action of a Lie group  by CR automorphisms.  He found that, up to a cover, every  such CR structure is a left-invariant CR structure  on a 3-dimensional Lie group  \cite[p.~69]{Ca1}. The items on this list form a  rich source of natural examples of CR geometries which, in our opinion, has   been hardly  explored and mostly forgotten. In this article  we present some of the most interesting  items on Cartan's list. 
 We outline Cartan's approach and, in particular,  the relation between the adjoint representation of the group and global realizability (the embedding of a CR structure as a hypersurface in a complex 2-dimensional manifold).

The spherical CR structure  on  $S^3$   can be thought of as the unique left-invariant CR structure on the group $\SUt\simeq S^3$ \hl{that} is also invariant by right translations by the standard diagonal circle subgroup $\mathrm{U}_1\subset \SUt$. There is a well-known and much studied 1-parameter family of deformations of this structure on $\SUt$ to structures whose only symmetries are  left translations by $\SUt$ (see, for example, \cite{Bu}, \cite{Cap},  \cite{CaMo}, \cite{Ro}). An  interesting feature of this family of deformations is that none of the structures, except the spherical one, can be globally realized as a hypersurface in $\C^2$  (although they can be  realized as finite covers
 of hypersurfaces  in 
$\CP^2$, the 3-dimensional orbits of the projectivization of the conjugation action of $\SUt$ on $\sltc$). This was  first shown in \cite{Ro} and later in  \cite{Bu} by a different and interesting proof; see Remark \ref{rmrk:burns} for a sketch of the latter proof. 

A left-invariant CR structure on a 3-dimensional Lie group $G$ is given by  a   1-dimensional complex subspace of its complexified Lie algebra $\g_\C$, that is, a point in the 2-dimensional   complex projective plane $\Pgc\simeq \CP^2$, satisfying a certain regularity condition (Definition \ref{def:reg} below). 
 The automorphism group of $G$, $\Aut(G)$, acts on the space of left-invariant CR structures on $G$, so that two $\Aut(G)$-equivalent  left-invariant CR structures on $G$  correspond to two points in $\Pgc$ in the same $\Aut(G)$-orbit. Thus the classification of left-invariant CR structures on $G$, up to CR-equivalence by the action of $\Aut(G)$, reduces to the classification of the  $\Aut(G)$-orbits in  $\Pgc$. This leaves the possibility that two left-invariant CR structures on $G$ which are not CR equivalent under $\Aut(G)$ might be still  CR-equivalent, locally or globally. Using  Cartan's equivalence method,  as introduced in \cite{Ca1}, we show in \hl{Theorem \ref{thm:ens}}
 that for {\em aspherical} left-invariant CR structures this possibility does not occur. Namely: two left-invariant aspherical CR structures on two 3-dimensional Lie groups  are CR equivalent if and only if the they are CR equivalent via a Lie group isomorphism. 
 See also \cite{BE} for a global invariant that distinguishes members of the left-invariant structures on $\SUt$ and Theorem 2.1 of \cite[p.~246]{ENS}, \hl{which  is the basis of our Theorem \ref{thm:ens}}.  The asphericity condition in Theorem \ref{thm:ens} is essential (see Remark \ref{rmrk:counter}).

\mn{\bf Contents of the paper.} In the next section, \S\ref{sec:prelim}, we  
present the basic definitions and  properties of CR manifolds. In \S\ref{sec:homog} we introduce some  tools  for studying  homogenous CR manifolds which will be used in later sections. 

In \S\ref{sec:sl2} we study our main   example of $G=\SLt$, where we find that up to $\Aut(G)$, there are two 1-parameter families of left-invariant CR structures, one {\em elliptic} and one {\em hyperbolic}, depending on the incidence relation of the associated contact distribution with the null cone of the Killing metric, see Proposition \ref{prop:p}. %
Realizations of these structures  are described in Proposition \ref{prop:slt2}: the  elliptic spherical structure can be  realized as any of the generic orbits of the standard representation in $\C^2$, or the complement of  $z_1=0$ in $S^3\subset\C^2$. The rest of the structures are  finite  covers of orbits of the  adjoint action   in  $\P(\sltc)=\CP^2$. The question of their global realizability in $\C^2$ remains open, as far as we know.

In \S\ref{sec:su2} we treat   the  simpler case of $G=\SUt$, where we recover the well-known 1-parameter family of left-invariant CR structures mentioned above, all with the same contact structure, containing a single spherical structure.  

The remaining two  sections present similar results for the Heisenberg and Euclidean groups.

In the Appendix we state the main differential geometric result of \cite{Ca1} and the specialization to homogeneous CR structures. 

\mn \centerline{*\qquad *\qquad*}

 \mn{\bf How `original' is this paper?} 
 We are  certain that \'Elie Cartan knew most  the results we present here. Some experts in his methods could likely  extract the {\em statements} of these results  from   his  paper \cite{Ca1}, where  Cartan  presents a classification of  homogeneous CR 3-manifolds in Chapter II.   As for finding the {\em proofs} of these results in \cite{Ca1}, or anywhere else, we are much less certain. The classification of homogeneous CR 3-manifolds appears on p.~70 of \cite{Ca1}, summing up more than 35 pages of general considerations  followed by case-by-case calculations. We found Cartan's text justifying the classification very hard to follow. The general ideas and techniques are quite clear, but we were unable to justify  many details of his  calculations and follow through the line of reasoning.  Furthermore, Cartan presents the classification in Chap.~II of \cite{Ca1}  before solving  the equivalence problem for CR manifolds in Chap.~III, so  the CR invariants  needed to distinguish the items on his list are not available, nor can he use the argument of our Theorem \ref{thm:ens}.  In spite of extensive search and consultations with several experts, we could not find anywhere in the  literature  a detailed and complete   statement in modern language of Cartan's  classification of homogeneous CR manifolds, let alone  proofs. We decided it would be more useful for us, and our readers,   to  abstain from further deciphering of \cite{Ca1} and to  rederive his classification.
  
  As for  \cite{ENS}, apparently the authors shared our frustration with Cartan's text, as they redo parts of the classification in a style similar to ours. But  we found their presentation sketchy and at times inadequate. For example, the reference on pp.~248 and 250 of  \cite{ENS} to the  `scalar curvature $R$ of the CR structure' is misleading.  There is no `scalar curvature'  in CR geometry.  Cartan's invariant called  $R$ is coframe dependent and so the formula given by the authors is meaningless without specifying the coframe used \hl{(which is not  provided).}  Also, the realizations they found for their  CR structures are rather different from ours.

 In summary, we lay no  claim for originality of the results of this paper. Our main purpose here is to give a new treatment  of an old subject.  We hope the reader will find it worthwhile.

 \mn{\bf Acknowledgments.} We thank Boris Kruglikov and Alexander Isaev for pointing out to   us the article  \cite{ENS}, on which  our Theorem \ref{thm:ens} is based. 
GB thanks Richard Montgomery and Luis Hern\'andez Lamoneda for useful conversations.  GB  acknowledges support from CONACyT under project 2017-2018-45886.

\section{Basic definitions and properties of CR manifolds}\label{sec:prelim} 
A {\em CR structure} on a 3-dimensional manifold $M$ is a rank 2 subbundle $\DS \subset TM$  together with an almost complex structure
$J$  on $\DS$, i.e.  a bundle automorphism $J:D\to D$ such that $J^2=-Id$.  The structure is {\em non-degenerate} 
if $\DS$  is a  contact structure, i.e.  its sections  bracket generate $TM$.  
We shall henceforth assume this non-degeneracy condition for all CR structures. We stress that in this article all  CR manifold are  assumed 3-dimensional and have an underlying contact structure.

A CR structure  is equivalently given by  a complex line subbundle $\VS \subset \DS_\C:=\DS\otimes\C,$ the $-i$ eigenspace of $J_\C:=J\otimes\C$, denoted also by $T^{(0,1)}M$.
Conversely, given a complex line subbundle $\VS \subset T_\C M:=TM\otimes\C$ such that $\VS \cap \overline \VS = \{0\}$ and  $\VS \oplus \overline \VS$ bracket generates $T_\C M$, there is a unique  CR structure $(\DS,J)$  on $M$ such that $V=T^{(0,1)}M$. A  section  of $\VS$ is a {\em complex vector field  of type}  $(0,1)$ and can be equally used to specify the CR structure, provided it is non-vanishing. 

A dual way of specifying a CR structure, particularly useful for calculations, is via an {\em adapted coframe. } This consists of a pair of 1-forms $(\phi,\phi_1)$ where $\phi$ is a real contact form, i.e.  $D=\Ker(\phi)$,  $\phi_1$ is a complex valued form of type $(1,0)$, i.e.  $\phi_1(Jv)=i\phi_1(v)$ for every $v\in D$,  and such that  $\phi\wedge\phi_1\wedge \bar\phi_1$ is non-vanishing. \hl{The line bundle} $V\subset T_\C M$ can then be recovered from $\phi, \phi_1$ as their common kernel. The non-degeneracy of $(D,J)$ is equivalent to the non-vanishing of $\phi\wedge \d \phi$. 
We will use in the sequel any of these equivalent definitions of a CR structure.

If $M$ is a real hypersurface in a complex 2-dimensional manifold $N$ there is  an induced CR structure  on $M$ defined by  $\DS := TM \cap \tilde J(TM)$, where $\tilde J$ is the almost complex structure on $N$, 
with the almost complex structure $J$ on $\DS$ given by the restriction of $\tilde J$   to $\DS$. Equivalently, $V=T^{(0,1)}M:= \left(T_\C M \right) \cap \left(T^{(0,1)}N\right)$.   A CR structure (locally) CR equivalent to a hypersurface in a complex 2-manifold  is called (locally) {\em realizable}.

Two CR manifolds $(M_i, D_i, J_i)$, $i=1,2$, are {\em CR equivalent} if there exists a diffeomorphism $f:M_1\to M_2$ such that $\d f(D_1)=D_2$ and such that $(\d f|_{D_1})\circ J_1=J_2\circ (\d f|_{D_1})$. Equivalently, $(\d f)_\C (V_1)=V_2.$ A {\em CR automorphism} of a CR manifold is a CR self-equivalence, i.e.  a diffeomorphism $f:M\to M$ such that $\d f$ preserves $D$ and $\d f|_D$ commutes with $J$.   
Local CR equivalence and automorphism are defined similarly, by restricting the above definitions to  open subsets. 
An {\em infinitesimal CR automorphism}  is a vector field  whose (local) flow acts by (local) CR automorphisms. Clearly, the set $\Aut_\CR(M)$ of CR automorphisms  forms a group under composition and   the set $\aut_\CR(M)$ of infinitesimal CR automorphisms forms a Lie algebra  under the Lie bracket of vector fields. In fact,   $\Aut_\CR(M)$  is naturally a Lie group of dimension $\leq \dim(\aut_\CR(M))\leq 8$, see Corollary \ref{cor:aut}  in the Appendix. 

The basic example  of CR structure is the  unit sphere $S^3=\{|z_1|^2+|z_2|^2=1\}\subset\C^2$ equipped with the CR structure induced from $\C^2$. Its group of CR automorphisms is the 8-dimensional simple Lie group $\PUto$. The action of the latter on $S^3$ is seen by embedding $\C^2$ as an affine chart in $\CP^2$, $(z_1, z_2)\mapsto [z_1:z_2: 1]$, mapping $S^3$ unto the hypersurface given in homogeneous coordinates by $ |Z_1|^2+|Z_2|^2=|Z_3|^2$,  the projectivized null cone of the hermitian  form $|Z_1|^2+|Z_2|^2-|Z_3|^2$ in $\C^3$ of signature $(2,1)$. The group $\U_{2,1}$ is the subgroup of $\GL_3(\C) $ leaving invariant this hermitian  form and its projectivized action  on  $\CP^2$ acts on $S^3$  by CR automorphism. It is in fact its {\em full} automorphism group. This is a consequence of the Cartan's equivalence method, see Corollary \ref{cor:aut}.

\mn 

Here are two standard  results of the general theory of CR manifolds. 
\begin{prop}[`Finite type' property] 
\label{prop:cont}
Let $M ,M'$ be two   CR manifolds with $M$ connected and $f:M\to M'$ a local CR-equivalence. Then $f$ is determined by its restriction to any open subset of $M$. In fact it is determined of its  2-jet  at a single  point of $M$. 
\end{prop}

\begin{proof}
The Cartan equivalence method associates canonically with each  CR 3-manifold $M$ a certain principal bundle $B\to M$ with 5-dimensional fiber, a  reduction of the bundle of second order frames on $M$, together with a canonical coframing of $B$ (an $e$-structure, or `parallelism'; see the Appendix for more details). Consequently, $f:M\to M'$  lifts to a bundle map  $\tilde f:B\to B'$ between the associated  bundles (in fact, the 2-jet of $f$, restricted to $B$), preserving the coframing. Now any  coframe  preserving map of coframed manifolds with a connected domain is determined by its value at a single point. 
Thus  $\tilde f$ is determined by its value at a single point  in $B$. It follows that $f$ is determined by its 2-jet at a single point in  $M$.   
\end{proof}

\begin{prop}[`Unique extension' property]\label {prop:ext} Let $f:U\to U'$ be a CR diffeomorphism between  open connected subsets  of $S^3$. Then $f$ can be extended uniquely to an element $g\in\Aut_{\CR}(S^3)=\PUto$.
\end{prop}

\begin{proof} Let $B\to S^3$ be the  Cartan bundle associated with the CR structure, as in the proof of the previous proposition, and $\tilde f:B|_{U}\to B|_{U'}$  the canonical lift of $f$. Since $\Aut_\CR(S^3)$ acts transitively on $B$ (in fact, freely, see Corollary \ref{cor:aut}),  for any given $p\in B|_{U}$  there is a unique  $g\in \Aut_\CR(S^3)$ such that $\tilde f(p)=\tilde g(p).$ It follows, by the previous proposition, that $f=g|_{U}$. See also \cite{A}, Proposition 2.1,  for a  different proof. \end{proof}

Here is a simple consequence  of the last two propositions    that  will be useful for us later. 

\begin{cor}\label{cor:sph}Let $M$ be a connected 3-manifold and $\phi_i:M\to S^3$, $i=1,2$, be two immersions. Then the two induced spherical CR structures on $M$ coincide if and only if $\phi_2=g\circ \phi_1$ for some $g\in\Aut_\CR(S^3)=\PUto$.

\end{cor}

\begin{proof}Let $U\subset M$ be  a connected open subset for which each  restriction 
$\left.\phi_i\right|_U$  is a  diffeomorphism unto its image $V_i:=\phi_i(U)\subset S^3$, 
$i=1,2$. Then $(\phi_2|_U)\circ(\phi_1|_U)^{-1}:V_1\to V_2$ is a CR diffeomorphism. By Proposition \ref{prop:ext}, there exists  $g\in \PUto$ such that $\phi_2|_U=(g\circ \phi_1)|_U.$ It follows, by Proposition \ref{prop:cont}, that $\phi_2 = g \circ \phi_1 .$ 
\end{proof}

\section{Left-invariant  CR structures on 3-dimensional Lie groups}\label{sec:homog}

A natural class of CR structures are the  {\em homogeneous} CR manifolds, i.e. CR manifolds admitting a transitive  group of automorphisms. 
Up to a cover, every such structure is given by a left-invariant CR structure on a 3-dimensional Lie group (see e.g. \cite[p.~69]{Ca1}). 
Each such Lie group is determined, again, up to a cover, by its Lie algebra. The list of possible Lie algebras is a certain sublist of the list of 3-dimensional real Lie algebras (the `Bianchi classification'), and was determined by \'E. Cartan in Chapter II of his 1932 paper \cite{Ca1}. 
In this section we first make some general remarks about such CR structures, then  state  an easy to apply criterion for sphericity. Our  main references here are Chapter II of \'E.~Cartan's paper \cite{Ca1} and \S 2 of Ehlers et al. \cite{ENS}.

\subsection{Preliminaries}
Let $G$ be a   3-dimensional Lie group $G$ with identity element $e$ and Lie algebra $\g=T_eG.$ To each $g\in G$ is associated the {\em  left translation} $G\to G$,  $x\mapsto gx$. A CR structure on $G$ is {\em left-invariant} if all left translations are  CR automorphisms. Clearly, a left-invariant CR structure $(\DS,J)$ is given uniquely by its value $(\DS_e,J_e)$ at $e$. Equivalently, it is given by a {\em non-real} 1-dimensional complex subspace $V_e\subset \g_\C:=\g\otimes\C$; i.e.  $V_e\cap \overline{V_e}=\{0\}$. By the  non-degeneracy of the CR structure,  $D_e\subset \g$ is not a Lie subalgebra; equivalently, $V_e\oplus \overline{V_e}\subset \g_\C$ is not a Lie subalgebra. In other words, {\em left-invariant CR structures are parametrized  by the non-real and  non-degenerate elements of  $\Pgc\simeq \CP^2$.} 
\begin{defn}\label{def:reg}
An element  $[L]\in \Pgc$ is  {\em real} if $[L]= [\overline L]$, {\em degenerate} if  $L,\overline L$ span a Lie subalgebra of $\g_\C$ \hl{and {\em regular} if it is neither real nor degenerate.}  The locus of regular elements in $\Pgc$ is denoted  by  $\Pgcr.$
\end{defn}

Equivalently, if $[L]=[L_1+iL_2]\in\Pgc$, where $L_1, L_2\in\g$, then $[L]$ is non-real if and only if $L_1, L_2$  are linearly independent and is regular  if and only if $L_1,L_2, [L_1,L_2]$ are linearly independent. 

\mn

Let $\Aut(G)$ be the group of Lie group automorphisms of $G$ and $\Aut(\g)$ the group of Lie algebra automorphisms  of $\g$. For each $f\in \Aut(G)$, $\d f(e)\in\Aut(\g)$, and if  $G$ is connected then  $f$ is determined uniquely by $\d f(e)$, so  $\Aut(G)$ embeds naturally as a  subgroup  $\Aut(G)\subset \Aut(\g)$. Every Lie algebra homomorphism of  a {\em simply connected} Lie group lifts uniquely  to a Lie group homomorphism, hence for simply connected $G$, $\Aut(G)=\Aut(\g)$.  The adjoint representation of $G$ defines a homomorphism $\Ad:G\to \Aut(G)$. Its image  is a normal subgroup  $\Inn(G)\subset\Aut(G)$, the  group of {\em inner} automorphisms (also called `the adjoint group'). The quotient group, $\Out(G):=\Aut(G)/\Inn(G)$, is the group of {\em outer} automorphisms. For a simple Lie group, $\Out(G)$ is a finite group. For example, $\Out(\SUt)$ is trivial and $\Out(\SLt)\simeq\Z_2,$ given by conjugation by any matrix $g\in \GL_2(\R)$ with negative determinant, e.g. $g=\diag(1,-1)$. 

Now $\Aut(G)$ clearly acts on the set of left-invariant CR structures on $G$. It also acts on  $\Pgcr$ by the  projectivized complexification of its action on $\g$. The map associating with a left-invariant CR structure $V\subset T_\C G$ the point $z=V_e\in\Pgcr$ is clearly  $\Aut(G)$-equivariant, hence if $z_1, z_2\in \Pgcr$  lie  on the same $\Aut(G)$-orbit  then  the corresponding  left-invariant  CR structures on $G$ are CR equivalent via an element of  $\Aut(G)$. As mentioned in the introduction, the converse is true   for  {\em aspherical} left-invariant CR structures.

\begin{thm}\label{thm:ens}
Consider two left-invariant aspherical CR structures  $V_i\subset T_\C G_i$ on two connected 3-dimensional Lie groups $G_i$, with corresponding  elements $z_i:=\left(V_i\right)_{e_i}\in \P((\g_i)_\C))_\mathrm{reg}$,  
where $e_i$ is the identity element of $G_i$, $i=1,2$. If the two CR structures are equivalent,  then there exists a group isomorphism $G_1\to G_2$ which is a CR equivalence, whose derivative at $e_1$ maps $z_1\mapsto z_2$. If the two CR structures are locally equivalent, then there exists a Lie algebra isomorphism $\g_1\to\g_2$, mapping  $z_1\mapsto z_2$. 
\end{thm}
\begin{proof}
Let  $f:G_1\to G_2$ be a  CR equivalence. By composing $f$ with an appropriate left translation, either
 in $G_1$ or  in $G_2$, we can assume, without loss of generality,  that $f(e_1)=e_2$. Since $f$ is a CR equivalence, $(\d f)_\C V_1=V_2$. In particular, $(\d f)_\C$ maps $z_1\mapsto  z_2$.   We next show that $f$ is a group isomorphism.

For any 3-dimensional Lie group $G$, the space  $\X(G)$ of right-invariant vector fields  is a 3-dimensional Lie subalgebra of the space  of vector fields on $G$, generating left-translations on $G$. Hence if $G$ is  equipped with a left-invariant CR structure then  $\X(G)\subset\aut_\CR(G).$ If the CR structure  is aspherical then the Cartan equivalence method  implies that $\dim(\aut_\CR(M))\leq 3$, see Corollary \ref{cor:aut} of the Appendix.  Thus   $\X(G)=\aut_\CR(G).$

Now since   $f:G_1\to G_2$ is a  CR equivalence, its derivative defines a Lie algebra isomorphism $\aut_\CR(G_1)\simeq
\aut_\CR(G_2)$. It follows, by the last paragraph, that $\d f(\X(G_1))=\X(G_2)$. This implies that $f$ is a group isomorphism by a result from the theory of Lie groups: If $f:G_1\to G_2$ is  a diffeomorphism  between two connected Lie groups such that $f(e_1)=e_2$ and $\d f(\X(G_1))=\X(G_2)$ then $f$ is a group isomorphism.

We could not  find a reference for the  (seemingly standard)  last statement  so we  sketch a proof here. 
Let $G=G_1\times G_2$ and $H=\{(x,f(x))| x\in G_1\}$ (the graph of $f$). Then $f$ is a group isomorphism if and only if $H\subset G$ is a subgroup. Let   $\h:=T_eH$, where $e=(e_1,e_2)\in G$, and  let $\mathcal{H}\subset TG$ the extension of $\h$ to a right-invariant sub-bundle. Then, since $\d f: \X(G_1)\to \X(G_2)$ is a Lie algebra isomorphism, $\h\subset \g$ is a Lie subalgebra, $\mathcal{H}$ is integrable   and $H$ is the integral leaf  of $\mathcal{H}$ through $ e\in G$   (a maximal connected integral submanifold of $\mathcal{H}$). It follows that $Hh$ is also an integral leaf of $\mathcal{H}$ for every $h\in H$. But $e\in H\cap Hh$, hence $H=Hh$ and so $H$ is closed under multiplication and inverse, as needed. 

To prove the last statement of the theorem, suppose  $f:U_1\to U_2$ is a  CR equivalence, where $U_i\subset G_i$ are  open subsets, $i=1,2$. By composing $f$ with appropriate left translations
 in $G_1$ and  $G_2$, we can assume, without loss of generality,  that $U_i$ is a neighborhood of $e_i\in G_i$, $i=1,2$,  and that
 $f(e_1)=e_2$. Since $f$ is a CR equivalence, its complexified derivative $(\d f)_\C:T_\C U_1\to T_\C U_2$ maps $V_1|_{U_1}$ isomorphically onto $ V_2|_{U_2}$; in particular, it  maps $z_1\mapsto z_2.$ It remains to show that $\d f(e_1):\g_1\to\g_2$ is a Lie algebra isomorphism.

For any Lie group $G$, the  Lie bracket of two elements   $X_e, Y_e\in \g=T_eG$ is defined  by evaluating at $e$
 the  commutator  $XY-YX$ of their unique  extensions to {\em left}--invariant vector fields $X,Y$ on $G$. If we use
  instead  {\em right}--invariant vector fields, we obtain the negative of the  standard Lie bracket. Now right-invariant 
  vector fields generate left translations, hence if $G$ is a 3-dimensional Lie group equipped with a left-invariant CR
   structure,  there is a natural inclusion  of Lie algebras $\g_-\subset \aut_\CR(G),$ where $\g_-$ denotes $\g$  equipped
    with the negative of the standard bracket. For any aspherical  CR structure on a 3-manifold $M$ we have
     $\dim(\aut_\CR(M))\leq 3$, hence for any open subset $U\subset G$ the restriction of a left-invariant aspherical CR 
     structure on $G$ to $U$ satisfies $\aut_\CR(U)=\X(G)|_U\simeq\g_-.$

Next,  since $f:U_1\to U_2$ is a CR equivalence, its derivative  $\d f$  defines 
 a Lie algebra isomorphism $\aut_\CR(U_1)\to\aut_\CR(U_2)$.  
By the previous paragraph,   $\d f(e)$ is a Lie algebra isomorphism  $(\g_1)_-\to (\g_2)_-$, and thus is also a Lie algebra isomorphism $\g_1\to\g_2$. 
\end{proof}

\subsection{A sphericity criterion  via well-adapted coframes}
We formulate  here a simple criterion for deciding whether a   left-invariant CR structure $z\in\Pgcr$ on a Lie group $G$ is spherical or not. 
The basic tools  are found in the seminal papers of Cartan \cite{Ca1},\cite{Ca2}.  We defer a more complete discussion to the Appendix.  

\begin{defn}\label{def:adap}
 Let $M$ be a 3-manifold with a CR structure $V\subset T_\C M$. An {\em adapted coframe}  is a pair of 1-forms $(\phi, \phi_1)$  with $\phi$ real and $\phi_1$ complex, such that  $\phi |_V=\phi_1 |_V = 0$ and $\phi\wedge \phi_1\wedge\bar\phi_1$ is non-vanishing.  The coframe is  {\em well-adapted} if $\d\phi = i\phi_1\wedge\bar{\phi_1}$. 
\end{defn}

Adapted and \wa coframes always exist, locally. Starting with an arbitrary
non-vanishing local section $L$ of $V$ (a complex vector field of type $(0,1)$) and 
a contact form $\theta$ (a non-vanishing local section of  $D^\perp\subset T^*M$), define
 the complex $(1,0)$-form $\phi_1$ by $\phi_1(L)=0$, $\bar \phi_1(L)=1$. Then 
 $(\phi, \phi_1)$ is an adapted coframe and any other adapted coframe is given by 
 $\tilde\phi=|\lambda|^2\phi,$ $\tilde \phi_1=\lambda(\phi +\mu \phi_1)$ 
 for arbitrary complex functions $\mu$,\ $\lambda$, with $\lambda$ non-vanishing. 
 It is then easy to verify that for any $\lambda$ and  $ \mu=i\,L(u)/u$ where   
 $ u=|\lambda|^2$, the resulting coframe $(\tilde \phi, \tilde\phi_1)$ is well-adapted.  

Given a well-adapted coframe $(\phi, \phi_1)$,  decomposing  $\d\phi, \d\phi_1$ in the same coframe we get  
\begin{align}\label{eq:str}
\begin{split}
\d\phi &= i\phi_1\wedge\bar{\phi_1}\\
\d\phi_1&=a\,\phi_1\wedge\bar \phi_1+b\,\phi\wedge \phi_1+c\,\phi\wedge \bar\phi_1,
\end{split}
\end{align}
for some complex valued functions $a,b,c$ on $M$. For a left-invariant CR structure on a 3-dimensional group $G$ one can choose a (global) well-adapted coframe of left-invariant 1-forms, and then $a,b,c$ are constants. 

%%%%%%%%%%%%%%%%%%%%%%%%%%%%%%%%%%%%%%%%%%%
\begin{prop}\label{prop:CRc}
Consider a   CR structure on a 3-manifold given by a   well adapted coframe $\phi, \phi_1$, satisfying equations \eqref{eq:str}  for some constants $a,b,c\in\C.$ The CR structure is spherical if and only if $c \left(2|a|^2+9 ib\right)=0.$
\end{prop}

This is a consequence of Cartan equivalence method. See Corollary \ref{cor:sph-curv} in the Appendix.

\subsection{Realizability } 
 Let $(M,D,J)$ be a CR 3-manifold and $N$ a complex manifold. A smooth function $f:M\to N$ is a {\em CR map}, or simply {\em CR}, if $\tilde J\circ (\d f|_D)=  (\d f|_D)\circ J$, where $\tilde J:TN\to TN$ is the almost complex structure on $N$. 
Equivalently, $(\d f )_\C V\subset T^{(0,1)}N.$ A {\em realization} of $(M,D,J)$  is a CR embedding of $M$ in a   (complex) 2-dimensional $N$. A {\em local realization} is a CR  immersion in such $N$. 

The following lemma  is useful for finding CR immersions and embeddings of  left-invariant CR structures on Lie groups. 

\begin{lemma}\label{lemma:real}Let $G$ be a 3-dimensional  Lie group with a left-invariant CR structure  $(D,J)$, with corresponding  $[L]\in\Pgcr$. Let $\rho:G\to \GL(U)$ be a finite dimensional complex representation, $u\in U$ and  $\mu:G\to U$  the evaluation map  $g\mapsto \rho(g)u$. Then $\mu$ is  a  CR map if and only if $\rho'(L)u=0,$ where $\rho':
\g_\C\to\End(U)$ is the  complex linear extension of $(\d\rho)_e:\g\to\End(U)$ to $\g_\C$. 
\end{lemma}
\begin{proof}
$\mu$ is clearly $G$-equivariant, hence $\mu$ is CR if  and only if $\d\mu( JX)=i\,\d\mu(X)$ for some (and thus all) non-zero $X\in D_e$. Now  $\d\mu(X)=\rho'(X)u,$ hence the CR condition on $\mu$ is $\rho'(X+iJX)u=0,$ for all $X\in D_e$. Equivalently, $\rho'(L)u=0$ for some (and thus all) non-zero $L\in \g_\C$ of type $(0,1)$. 
\end{proof}

Here is an application of the last  lemma, often  used by Cartan in Chapter II of \cite{Ca1}.

\begin{prop}\label{prop:real}
Let $G$ be a 3-dimensional  Lie group with a left-invariant CR structure   $[L]\in\Pgcr$. Then the evaluation map  $\mu:G\to\Pgc$, $g\mapsto [\Ad_g(L)]$, is  a $G$-equivariant CR map, whose image $\mu(G)\subset\Pgc$, the $\Ad_G$-orbit of $[L]\in\Pgc$,  is  of dimension 2 or 3. It follows that if $L$ has a trivial centralizer in $\g$ then  $\mu(G)$ is 3-dimensional and hence  $\mu$ is a local realization  of the CR structure on $G$  in $\Pgc\simeq\CP^2$. 
\end{prop}

\begin{proof}
Let $\tilde\mu: G\to \g_\C\setminus\{0\},$ $g\mapsto \Ad_g L,$
and $\pi: \g_\C\setminus\{0\}\to\Pgc$, $B\mapsto [B]$. 
Then $\mu=\pi\circ\tilde\mu$ and $\pi$ is holomorphic, 
hence it is enough to show that $\tilde\mu$ is CR at $e\in G$. Applying Lemma 
 \ref{lemma:real} with $\rho=\Ad_G$, $u=L$, we have that $\rho'(L)L=[L,L]=0$, hence $\tilde\mu$ is CR, and so is $\mu.$  
 
Let  $\OO=\mu(G)$.  Since $\mu$ is CR, $\d\mu(D)$ is a $\tilde J$-invariant and $G$-invariant subbundle of $T\OO$, where $\tilde J$ is the almost complex structure of $\Pgc$. Thus in order to show that $\dim(\OO)\geq 2$  it is enough to show that $\d\mu(D_e)\neq 0$. 
  Equivalently, $\d\tilde\mu (D_e)\not\subset \Ker((\d\pi)_L)=\C L$.  Let $L=L_1+iL_2,$ with $L_1, L_2\in \g$. Then $L_2=JL_1$ and   so $\d\tilde\mu ( L_2)=[L_2,L]=-[L_1,L_2]$. But $[L]$ is non-real, so $(\C L)\cap \g=\{0\}$, hence  $[L_1,L_2]\in\C L$ implies  $[L_1,L_2]=0$, so $D_e=\Span\{L_1, L_2\}\subset \g$ is an (abelian) subalgebra,  in contradiction to the non-degeneracy assumption on the CR structure. 
 \end{proof}

\section{$\SLt$}\label{sec:sl2}

We illustrate the results of the previous section first of all with a detailed description of left-invariant CR structures on  the group $G=\SLt$, where $\g=\slt$, the set of $2\times 2$ traceless real matrices and $\g_\C=\sltc$, the set of $2\times 2$ traceless complex matrices. 

Here is a summary of the results:  for $G=\SLt$, the set of left-invariant CR structures $\Pgcr$ is identified $\Aut(G)$-equivariantly with the set  of unordered pairs of points $\zeta_1, \zeta_2\in\C\setminus\R$,  $\zeta_1\neq \bar\zeta_2$, on which  $\Aut(G)$ acts   by orientation preserving isometries of the  usual hyperbolic  metric in each of the half. With this description, it is easy to determine  the $\Aut(G)$-orbits. There are two families of orbits: the  `elliptic' family corresponds to pairs of points  in the same half-plane,  with the spherical structure corresponding to a `double point', $\zeta_1=\zeta_2$;  the `hyperbolic' family corresponds to non-conjugate pairs of points   in opposite half planes.
 Each orbit is labeled uniquely by the hyperbolic distance $d(\zeta_1,\zeta_2)$ in the elliptic case, or $d(\zeta_1, \bar \zeta_2)$  in the hyperbolic case. All structures, except the spherical elliptic one, are locally realized as   adjoint orbits in  $\Psltc=\CP^2$, either inside $S^3=\{[L]\st \tr(L\bar L)=0\}$ (in  the hyperbolic case) or in its exterior (in the elliptic case). The  elliptic spherical structure  embeds as any of the  generic orbits of the standard action on $\C^2$. 

\sn

We begin with the conjugation action of $\SLtc$ on $\Psltc$ (this will be useful also for the next example of $G=\SUt$).  With each $[L]\in\Psltc$ we associate 
an unordered  pair of points $\zeta_1, \zeta_2\in\C\cup\infty,$ possibly repeated, the roots of the quadratic polynomial 
\be\label{eq:p} 
p_L(\zeta):=c\zeta^2-2a\zeta-b=c(\zeta-\zeta_1)(\zeta-\zeta_2), \qquad L=\left(
\begin{array}{rr}
a &b\\
c&-a
\end{array}\right).
\ee
Clearly, multiplying $L$ by a non-zero complex constant does not affect  $\zeta_1, \zeta_2$. 

\begin{lemma}\label{lemma:equiv}
Let $S^2(\CP^1)$ be the set of unordered pairs of points  $\zeta_1, \zeta_2\in \C\cup\infty=\CP^1$. 
 Then: 
 \benum
 \item  The map $\P(\sltc)\to S^2(\CP^1)$, assigning to $[L] \in \P(\sltc)$ the roots of $p_L$, as in equation \eqref{eq:p}, is an 
$\SLtc$-equivariant bijection, where $\SLtc$ acts on $S^2(\CP^1)$ via M\"obius 
transformations on $\CP^1$ (projectivization of the standard action on $\C^2$);

 \item   Complex  conjugation,  $[L]\mapsto [\overline L]$, corresponds, under the above bijection,  to complex conjugation of the roots of $p_L$, $\{\zeta_1,\zeta_2\}\mapsto \{\bar\zeta_1, \bar\zeta_2\}$. 
\end{enumerate}
\end{lemma}
\begin{proof}The map $[L]\mapsto  \{\bar\zeta_1, \bar\zeta_2\}$  is clearly a bijection (a polynomial is determined, up to a scalar multiple, by its roots). The $\SLtc$-equivariance, as well as item (b),  can be easily checked by direct computation.

Here is a more illuminating  argument, explaining also  the origin of the formula  for $p_L$ in equation \eqref{eq:p}. We first show that the adjoint  representation of $\SLtc$ on $\sltc$ is isomorphic to $H_2$, the space of quadratic forms on $\C^2$, or complex homogeneous polynomials $q(z_1, z_2)$ of degree 2 in two variables, with $g\in \SLtc$ acting by substitutions, $q\mapsto q\circ g^{-1}$. To derive an explicit isomorphism,  
let $\U$ be the standard  representation of $\SLtc$ on $\C^2$ and $\U^*$ the dual representation, where $g\in\SLtc$ acts on $\alpha\in \U^*$ by $\alpha\mapsto \alpha\circ g^{-1}$. The induced action on $\Lambda^2(\U^*)$ (skew symmetric bilinear forms on $\U$) is trivial (this amounts  to $\det(g)=1$). Let us fix $\omega:=z_1\wedge z_2\in\Lambda^2(\U^*)$. Since $\omega$ is  $\SLtc$-invariant,  it defines an 
$\SLtc$-equivariant isomorphism $\U\to \U^*$, $u\mapsto \omega(\cdot, u),$ mapping 
$\e_1\mapsto -z_2,$ $\e_2\mapsto z_1$, where $\e_1,\e_2$ is the standard basis of $\U$, dual to $z_1, z_2\in \U^*$. 
We thus obtain an   isomorphism of $\SLtc$ representations, $\End(\U)\simeq
\U\otimes \U^*\simeq \U^*\otimes \U^*$. Under this isomorphism,  $\sltc\subset\End(\U)$ is mapped unto $S^2(\U^*)\subset \U^*\otimes \U^*$ (symmetric bilinear forms on $\U$), which in turn is identified with $H_2$, $\SLtc$-equivariantly,  via $B\mapsto q$, $q(u)=B(u,u).$ Following through these isomorphisms, we get the sought for  $\SLtc$-equivariant isomorphism $\sltc\iso H_2$, 
\begin{align*}
L=\left(
\begin{array}{rr}
a &b\\
c&-a
\end{array}\right) &\mapsto
a \e_1\otimes z_1+b \e_1\otimes z_2+c \e_2\otimes z_1-a \e_2\otimes z_2\\
&\mapsto 
-a z_2\otimes z_1-b z_2\otimes z_2+c z_1\otimes z_1-a z_1\otimes z_2\\
&\mapsto
q_L(z_1,z_2)=c(z_1)^2-2a\,z_1z_2-b(z_2)^2.
\end{align*}
Now every non-zero  quadratic form $q\in H_2$ can be factored as the product of two non-zero linear forms, $q=\alpha_1\alpha_2$, where the kernel of each $\alpha_i$ determines a `root' $\zeta_i\in\CP^1$. Introducing the inhomogeneous coordinate $\zeta=z_1/z_2$ on $\CP^1=\C\cup\infty$, we get $c(z_1)^2-2a\,z_1z_2-b(z_2)^2=(z_2)^2p_L(\zeta)$, with  $p_L$  as in equation \eqref{eq:p} with roots $\zeta_i\in \C\cup\infty.$  
\end{proof}

\begin{rmrk}\label{rmrk:proj} There is a simple projective  geometric interpretation of Lemma \ref{lemma:equiv}. See Figure \ref{fig:proj}(a). 
Consider in the projective plane $\Psltc\simeq\CP^2$ the conic $\CC:=\{[L]\st \det(L)=0\}\simeq \CP^1.$ 
Through a point $[L]\in\CP^2\setminus\CC$ pass two (projective) lines tangent to $\CC$, with tangency points $\zeta_1,\zeta_2\in\CC$ (if  $[L]\in\CC$ then $\zeta_1=\zeta_2=[L]$). Since $\SLtc$ acts on $\CP^2$ by projective transformations preserving $\CC$, the map $[L]\mapsto \{\zeta_1, \zeta_2\}$ is  $\SLtc$-equivariant.   The map $[L]\mapsto [\overline L]$ is the reflection about $\RPt\subset\CP^2.$  Formula \eqref{eq:p}  is a coordinate expression of this geometric recipe. 
\end{rmrk}

\begin{figure}[h!]
\centerline{\includegraphics[width=\textwidth]{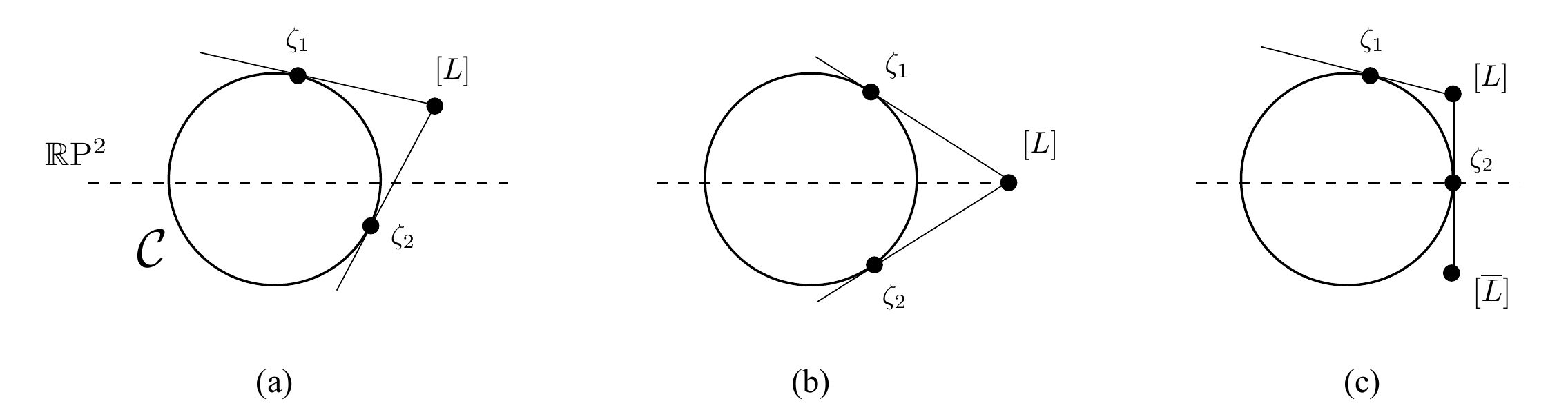}}
\caption{Distinct types of $[L]\in\Pgc$ for $G=\SLt$: (a)  regular ; (b) real  ; (c)  non-real degenerate. See the proofs of Lemma \ref{lemma:equiv}, \ref{lemma:reg} and   Remark \ref{rmrk:proj}. 
}\label{fig:proj}
\end{figure}

\begin{lemma}\label{lemma:reg} Let $L\in\sltc$, $L\neq 0$. Then $[L]\in\Psltcr$ if and only if both roots of $p_L$ are non-real and are non-conjugate, i.e.  $\zeta_1, \zeta_2\in\C\setminus\R$ and $\zeta_1\neq \bar\zeta_2$. 
 \end{lemma}
 
 \begin{proof}
Let $\zeta_1, \zeta_2$ be the roots of $p_L$. 
By Lemma \ref{lemma:equiv} part (b), $[L]$ is real,  $[L]= [\overline L]$, if and only if $\zeta_1, \zeta_2$ are  both real or $\zeta_1=\bar\zeta_2$. 
We claim that if $[L]\neq [\overline L]$ then $[L]$ is degenerate, i.e.  $L,\overline L$ span a 2-dimensional subalgebra of $\sltc$, exactly when  one of the two roots $\zeta_1, \zeta_2$ is real and the other is non-real. This is perhaps best seen with Figure \ref{fig:proj}(c). A 2-dimensional subspace of $\sltc$ corresponds to a projective line in $\Psltc$. The 2-dimensional subalgebras of $\sltc$ are all conjugate (by $\SLtc$) to the subalgebra of upper triangular matrices and are represented in  Figure \ref{fig:proj} by lines tangent to $\CC$. Now the line passing through $[L], [\overline L]$ is invariant under complex conjugation, hence  if it is tangent to to $\CC$ then the tangency point is real and is one of the roots of $p_L$. But $[L]$ is non-real, hence the other root is non-real.  
\end{proof}

Next we describe $\Aut(\SLt)$. Clearly, $\GL_2(\R)$ acts on $\SLt$ by matrix conjugation as group automorphism. The ineffective kernel of this action  is the center  $\R^*\II$ of $\GLt$ (non-zero multiples of the identity matrix). The quotient group is denoted by $\PGL_2(\R)=\GLt/\R^*\II.$ Thus there is a natural inclusion $\PGL_2(\R)\subset \Aut(\SLt)$.   

\begin{lemma}$\PGLt= \Aut(\SLt)=\Aut(\slt)$.   
\end{lemma}
\begin{proof} We have already seen the  inclusions  $\PGLt\subset\Aut(\SLt)\subset \Aut(\slt)$, so it is enough to show that $\Aut(\slt)\subset\PGLt.$ Now the Killing form of a Lie algebra, 
$\<X,Y\>=\tr(\ad X\circ\ad Y)$,  is defined  in terms of the Lie bracket alone. For $\slt$, the associated quadratic form is $\det(X)=-a^2-bc$ (up to a constant), a non-degenerate quadratic form of signature (2,1).  Furthermore, the   `triple product' $(X,Y,Z)\mapsto \< X, [Y,Z]\>$ defines a non vanishing volume form on $\slt$ in terms of the Lie bracket, hence $\Aut(\slt)\subset\SO_{2,1}.$ Finally,  $\PGLt\subset \SO_{2,1}$ and both are 3-dimensional groups with two components, so  they must coincide. \end{proof}

Let us now examine the action of $\Aut(\SLt)$ on $\P(\sltc).$ It is convenient, instead of working with $\Aut(\SLt)=\PGL_2(\R)$, to work with its double cover $\SL^\pm_2(\R)$ (matrices with $\det=\pm 1.$) The latter consists of two components, the identity component, $\SLt$, and $\sigma\SLt$, where $\sigma$ is any matrix with $\det=-1$; for example $\sigma=\diag(1,-1)$. 
According to Lemma \ref{lemma:equiv}, we need to consider first the action of $\SL^\pm_2(\R)$ by M\"obius transformations on $\CP^1$. The action of the identity component $\SLt$  has 3 orbits; in terms of the inhomogeneous coordinate $\zeta$,  these are 
\begin{itemize}
\item  the upper half-plane $\Im(\zeta) > 0,$ 
\item  the lower half-plane $\Im(\zeta)< 0,$ 
\item  their common boundary, the real projective line $\RP^1 =\R\cup\infty$. 
\end{itemize}
The action on each half-plane is by orientation preserving hyperbolic isometries (isometries of the Poincar\'e metric $|\d \zeta|/|\Im(\zeta)|$). The action of $\sigma=\diag(1,-1)$ is by reflection about the origin $\zeta=0$, an orientation preserving hyperbolic isometry between the upper and lower half planes.

In summary, we get the  following orbit structure:

\begin{prop}\label{prop:p} Under the identification $\Psltc\simeq S^2(\CP^1)$ of Lemma \ref{lemma:equiv}, the  orbits of $\Aut(\SLt)$  in $\Psltcr$  correspond to the following two 1-parameter families of orbits in $S^2(\CP^1)$:
\begin{enumerate}[leftmargin=18pt,label=I.]\setlength\itemsep{5pt}
\item A 1-parameter family of orbits, corresponding to a pair of points $\zeta_1, \zeta_2\in\C\setminus\R$ in the same  half-plane  (upper or lower). The parameter can be taken as the hyperbolic distance $d(\zeta_1, \zeta_2)\in[0,\infty)$. All these orbits are 3-dimensional, except the one corresponding to a double point $\zeta_1=\zeta_2$, which is 2-dimensional.

\item A 1-parameter family of orbits, corresponding to pair of points   $\zeta_1, \zeta_2\in\C\setminus\R$ situated  in opposite half planes and which are not complex conjugate, $\zeta_1\neq \bar \zeta_2.$ The parameter can be taken as the hyperbolic distance $d(\zeta_1, \bar \zeta_2)\in(0,\infty)$. All these orbits are 3-dimensional. \end{enumerate}

The rest of the orbits are either real ($\zeta_1, \zeta_2\in \RP^1=\R\cup \infty$ or $\zeta_1= \bar \zeta_2$) or degenerate (one of the points is real). 
\end{prop}

\begin{proof} Most of the claims follow immediately from the previous lemmas so their proof is omitted. The claimed dimensions of the orbits follow from the dimension of the stabilizer in $\Aut(\SLt)$ of an unordered pair  $\zeta_1, \zeta_2\in \C\setminus\R$; for two distinct points in the same half-plane, or in opposite half-planes with $z_1\neq \bar z_2$ , the stabilizer is the two element subgroup interchanging the points. For a double point the stabilizer is a  circle group of hyperbolic rotations about this point. \end{proof}

Next, recall that the {\em Killing form}  on $\slt$ is the bilinear form $\<X,Y\>=(1/2)\tr(XY).$
The associated quadratic form $\<X, X\>=-\det(X)=a^2+bc$  is a non-degenerate indefinite form of signature $(2,1)$, the unique $\Ad$-invariant form on $\slt$, up to scalar multiple. The {\em null cone} $C\subset \slt$ is the subset of elements with  $\<X, X\>=0.$ 

\begin{defn} A 2-dimensional subspace $\Pi\subset\slt$  is called {\em elliptic} (respectively,{\em hyperbolic}) if the Killing form restricts to a definite (respectively, indefinite, but non-degenerate) inner product on $\Pi$. Equivalently, $\Pi$ is hyperbolic  if its  intersection with  the null cone $C$ consists of two of its generators and elliptic if it intersects it only at its vertex $X=0$. A left-invariant CR structure $(D,J)$ on $\SLt$ is elliptic  (resp.~{\em hyperbolic}) if $D_e\subset \slt$ is elliptic  (resp.~{\em hyperbolic}). 
\end{defn}

\begin{rmrk} There is a third type of a 2-dimensional subspace $\Pi\subset\slt$,  called {\em parabolic}, consisting of 2-planes tangent to $C$, but these are subalgebras of $\slt$, hence are excluded by the non-degeneracy condition on the CR structure.
\end{rmrk}

\begin{rmrk}
Our use of the terms elliptic and hyperbolic for the contact plane is natural from the point of view of Lie theory.  However it conflicts with the terminology of analysis; CR vector fields are never elliptic or hyperbolic differential operators.
\end{rmrk}

\begin{lemma}\label{lemma:elip}
Let $[L]\in\Psltcr$, and $D_e\subset \slt$ the real part of the span of $L, \overline L$. Then $D_e$ is elliptic if the  roots of $p_L$ lie in the same half plane (type I of Proposition \ref{prop:p}), and is hyperbolic if they lie in opposite half planes  (type II of proposition \ref{prop:p}). 
\end{lemma}
\begin{proof} Let $\zeta_1, \zeta_2$ be the roots of $p_L$. Acting by $\Aut(\SLt)$, we can assume, without loss of generality, that $\zeta_1=i$ and $\zeta_2=it$ for some $t\in\R\setminus\{-1,0\}$. Thus, up to scalar multiple,  $p_L=(\zeta-i)(\zeta-it)=\zeta^2-i(1+t)\zeta-t.$ A short calculation shows that  $D_e$ consists of matrices of the form 
$X=\left(\begin{array}{cc}
a(1+t) &tb\\
b& -a(1+t)
\end{array}
\right)$, $a,b\in\R$, with $\det (X)=-a^2(1+t)^2-tb^2.$ This is negative definite for $t>0$ and indefinite otherwise. 
\end{proof}

\begin{prop}\label{prop:slt1} Let $V_t\subset T_\C\SLt$, $t\in\R,$  be the left-invariant complex line bundle spanned at $e\in\SLt$ by  
\be\label{eq:L}
L_t=\left(\begin{array}{cc}
i{1+t\over 2} &t\\
1& -i{1+t\over 2} 
\end{array}
\right)\in\slt\otimes\C=\sl_2(\C).\ee

 Then 
\benum
\item 
$V_t$ is a left-invariant CR structure for all $t\neq 0,-1$, elliptic for $t>0$ and hyperbolic for  $t<0, t\neq -1.$ 

\item $V_t$ is spherical if  $t=1$ or $-3\pm  2\sqrt{2}$ and aspherical otherwise. 

\item Every  left-invariant  CR structure on $\SL_2(\R)$  is CR equivalent to 
 $V_t$ for  a unique $t\in(-1,0)\cup(0,1].$

\item The  aspherical left-invariant CR structures $V_t$, $t\in(-1,1)\setminus\{0,-3+  2\sqrt{2} \}$, are  pairwise non-equivalent, even locally.  

\end{enumerate}
\end{prop}

\begin{proof}

(a) The quadratic polynomial corresponding to $L_t$ 
is 
$$p(\zeta)=\zeta^2-i(1+t)\zeta-t=(\zeta-i)(\zeta-it),$$
with roots $i,it$. 
For  $t>0$ the  roots are in the upper half plane and thus, by  Lemma \ref{lemma:elip}, $V_t$ is an elliptic 
   CR structure. For $t<0$ the roots are in opposite half planes and 
   for $t\neq -1$ are not complex conjugate, hence $V_t$ is an hyperbolic  CR structure. 

\sn (b) Let 
$$\Theta=g^{-1}\d g= \left(\begin{array}{lr}
\alpha &\beta\\
\gamma & -\alpha
\end{array}
\right)
$$
be the left-invariant Maurer-Cartan $\sl_2(\R)$-valued 1-form on $\SL_2(\R)$. A coframe adapted to  $V_t$  is  
\be\label{eq:adap}
\theta=\beta-t\gamma, 
\quad \theta_1=\alpha-i{1+t\over 2}\gamma,
\ee
i.e.  $\theta(L_t)=\theta_1(L_t)=0,$ $\bar\theta_1(L_t)\neq 0$.  The Maurer-Cartan equations, $\d\Theta=-\Theta\wedge\Theta$,  are
$$\d\alpha=-\beta\wedge\gamma,\quad \d\beta=-2\alpha\wedge\beta,\quad\d\gamma=2\alpha\wedge\gamma.
$$
Using there equations, we calculate  
$$\d\theta=i{4t\over 1+t}\theta_1\wedge\bar\theta_1+\theta\wedge\theta_1+\theta\wedge\bar\theta_1.
$$
Now 
$$\phi:=\sign(t)(\beta-t\gamma), \quad \phi_1:=\sqrt{\left|{4t\over 1+t}\right|}\left[\alpha-i {1+t\over 4}\left({\beta\over t}+\gamma\right)\right]$$
satisfy
$$\d\phi =i\phi _1\wedge \bar{\phi_1},
\quad \d \phi_1=b\phi\wedge\phi_1+c\phi\wedge\bar\phi_1,$$
where
$$
b=-i{1+6t+t^2\over 4|t|(1+t)}, 
\quad c=-i{(1-t)^2\over 4|t|(1+t)}, 
$$
thus $(\phi, \phi_1)$ is well-adapted to $V_t$. Applying Proposition \ref{prop:CRc}, we conclude that $V_t$  is spherical if and only if $(1+6t+t^2)(1-t)=0;$ that is, $t=1$ or $-3\pm  2\sqrt{2}$, as claimed. 

\sn(c)  The hyperbolic distance $d(i, it)$ varies monotonically from $0$ to $\infty$ as $t$ varies from $1$ to $0$, hence every pair of points in the same  half plane can be mapped by $\Aut(\SLt)$ to the pair  $(i,it)$ for a unique $t\in (0,1]$. Consequently, every left-invariant elliptic CR structure is CR equivalent to  $V_t$ for a unique $t\in(0,1]$. 

 Similarly, $d(i, -it)$ varies monotonically  from $0$ to $\infty$
as $t$ varies from $-1$ to $0$, hence every  hyperbolic  left-invariant CR structure  is CR equivalent to $V_t$ for  a unique  $t\in(-1,0)$. 

By Theorem \ref{thm:ens}, no pair of the aspherical   $V_t$ with $0<|t|<1$ are  CR equivalent, even locally. It remains to show that the elliptic and hyperbolic spherical  structures, namely,  $V_t$ for $t=1$ and $-3+2\sqrt{2}$ (respectively), are not CR equivalent.  In the next proposition, we find an embedding $\phi_1:\SLt\to S^3$ of the elliptic spherical structure in the standard spherical CR structure on $S^3$ and an immersion $\phi_2:\SLt\to S^3$ of the hyperbolic spherical structure which is not an embedding (it is a $2:1$ cover). It follows from  Corollary \ref{cor:sph} that these two spherical structures are not equivalent: if $f:\SLt\to\SLt$ were  a  diffeomorphism  mapping the hyperbolic spherical structure to the elliptic one,  then this  would imply that the pull-backs to $\SLt$  of the spherical structure of $S^3$ by $ \phi_1\circ f$ and $\phi_2$ coincide,  and hence,  by Corollary \ref{cor:sph}, there is an element  $g\in\PUto$ such that  $\phi_2=g\circ \phi_1\circ f$. But this is impossible, since $g\circ \phi_1\circ f$ is an embedding and $\phi_2$ is not.

\sn (d) As mentioned in the previous item, this is a consequence of Theorem \ref{thm:ens}. 
\end{proof}

\begin{rmrk}\label{rmrk:short} There is an alternative path, somewhat shorter (albeit less picturesque),  to the classification of left-invariant 
CR structures on $\SLt$, leading to a family of `normal forms' different then the  $V_t$ of 
Proposition \ref{prop:slt1}. One shows first that, up to conjugation by $\SLt$, there are only  
two non-degenerate left-invariant contact structures  $D\subset T\SLt$: an elliptic one, given by $D_e^+=\{c=b\},$ 
and hyperbolic one, given by $D_e^-=\{c=-b\}.$ The Killing form on $\slt$, $-\det(X)=a^2+bc$, restricted to $D_e^\pm$, is given by $a^2\pm b^2, $ with orthonormal basis $A,B\pm C$, where $A, B, C$ is the basis of $\slt$ dual to $a,b,c$. 
One then determines the stabilizer of  $D_e^\pm$ in  $\Aut(\SLt)$ 
(the subgroup that leaves $D_e^\pm$ invariant).  In each case the stabilizer acts on $D_e^\pm$ as the full isometry group of $a^2\pm b^2,$ that is, $\rm{O}_2$ in the elliptic case,  and $\rm{O}_{1,1},$ in the hyperbolic case. Using this description one shows that, in the elliptic case, each almost complex structure on $D^+_e$ is conjugate to a unique one of the form  $A\mapsto s(B+C)$, $s\geq 1$, with corresponding $(0,1)$ vector    $A+is(B+C)=\left(\begin{array}{cc}
1 &is\\
is&-1 
\end{array}
\right)$, and in the hyperbolic case $A\mapsto s(B-C)$, $s> 0$,  with corresponding $(0,1)$ vector    $A+is(B-C)=\left(\begin{array}{cc}
1 &is\\
-is&-1 
\end{array}
\right)$.   The spherical structures are given by $s=1$ in both cases.  
\end{rmrk}

Regarding realizability of left-invariant CR structures  on $\SLt$,  we have the following.

\begin{prop}\label{prop:slt2}

\benum
\item
The elliptic left-invariant spherical CR structure  on $\SLt$   ($t=1$ in equation \eqref{eq:L}) is realizable as  any of the generic (3-dimensional) $\SLt$-orbits in  $\C^2$ (complexification of the standard linear action on $\R^2$). This is also CR equivalent to the complement of a `chain'  in $S^3\subset \C^2$ (a curve in $S^3$ given by the intersection of a complex affine line in $\C^2$ with $S^3$; e.g. $z_1=0$)

\item The rest of the   left-invariant CR structures on $\SLt$, with $0<|t|<1$ in equation \eqref{eq:L},  are either $4:1$ covers, in the aspherical elliptic case $0<t<1$, or $2:1$ covers, in the hyperbolic case $-1<t<0$, of  the orbits of  $\SLt$ in  $\Psltc$. 

\item The spherical hyperbolic orbit  is also CR equivalent to the complement of $S^3\cap \R^2$ in $S^3\subset\C^2$. 
\end{enumerate}
\end{prop}

\begin{proof}

(a)  Fix $v\in\C^2$ and define $\mu:\SLt\to\C^2$ by $\mu(g)= gv$. If the stabilizer of $v$ in $\SLt$ is trivial and $L_1v=0,$ then, by Lemma  \ref{lemma:real}, $\mu$ is an $\SLt$-equivariant CR embedding. 
Both conditions are satisfied by $v={i\choose 1}.$  In fact,  all  3-dimensional $\SLt$-orbits in $\C^2$ are  homothetic, hence are CR equivalent and any of them will do. 

Now let $\OO\subset\C^2$ be the $\SLt$-orbit of $v={i\choose 1}.$ For $g={a\ \ b\choose c \ \ d}\in \SLt$, with $\det(g)=ad-bc=1$, $gv={b+ia\choose d+ic}$, hence $\OO$ is the quadric $\Im(z_1\bar z_2)=1$, where $z_1, z_2$ are the standard complex coordinates in $\C^2$. To  map $\OO$ onto the complement of $z_1=0$ in $S^3$ we first apply the  complex linear transformation $\C^2\to \C^2$, $(z_1, z_2)\mapsto (z_1+iz_2, z_2+iz_1)/2$, mapping $\OO$ unto the hypersurface $|z_1|^2-|z_2|^2=1$. 
Next  let $Z_1, Z_2, Z_3$  be homogenous coordinates in  $\CP^2$ and embed $\C^2$ as an affine chart,   $(z_1, z_2)\mapsto [z_1: z_2:1].$ The image  of $|z_1|^2-|z_2|^2=1$ is the complement of $Z_3=0$ in $|Z_1|^2-|Z_2|^2=|Z_3|^2.$ This is mapped by  $[Z_1:Z_2:Z_3]\mapsto [Z_3: Z_2: Z_1]$ to the complement of $Z_1=0$ in $|Z_1|^2+|Z_2|^2=|Z_3|^2.$ In our affine chart   this  is the complement of $z_1=0$ in  $|z_1|^2+|z_2|^2=1$, as needed.

\sn (b) By Proposition \ref{prop:real}, to show that the map $\SLt\to\Psltc$, $g\mapsto [\Ad_gL_t]$, is a CR immersion of  $V_t$ into $\Psltc$, it is enough to to show that the  stabilizer of $[L_t]\in\Psltc$ in $\SLt$ is discrete. Using  Lemma \ref{lemma:equiv}, we find that, in the aspherical elliptic case, where  $t\in (0,1)$, the roots are an unordered pair of distinct  points in the upper half plane, so there is a single hyperbolic isometry in $\PSLt$ interchanging them, hence the stabilizer in $\SLt$ is a 4 element subgroup. 

In  the hyperbolic case, where  $t\in (-1,0)$,  the roots $\zeta_1, \zeta_2$ are in opposite half-spaces and $\zeta_1\neq \bar\zeta_2$. Hence an element $g\in\SLt$ that fixes  the unordered pair $\zeta_1, \zeta_2$ has  two distinct fixed points   $\zeta_1, \bar \zeta_2$ in the same half plane. It follows that $g$ acts trivially in this half plane and thus $g=\pm \II.$ 

\sn (c) $\sltc$ admits a  pseudo-hermitian product of signature $(2,1)$, $\tr\left(X\overline Y\right)$, invariant under the conjugation action of $\SLt$. The associated projectivized null cone in $\CP^2$ is diffeomorphic to $S^3$, a model for the spherical CR structure on $S^3$. One can check that $L_t$ is a null vector, i.e.  $\tr(L_t\bar L_t)=0$,   for $t=-3\pm\sqrt{2}$. Thus  the hyperbolic spherical left-invariant structure on $\SLt$ is  a $2:1$ cover  of   an  $\SLt$-orbit  in $S^3$, consisting of all regular elements $[L]\in S^3$, whose  complement  in $S^3$ is  the set of  elements which are either real or degenerate non-real (see Lemma \ref{lemma:reg} and its proof). One can check that the only degenerate element in $S^3$ satisfies $a=c=0,$ $b\neq 0$, which is real. Thus  all irregular elements in $S^3$ are  the  real elements  $\RP^2\cap S^3\subset \CP^2,$ or $\R^2\cap S^3\subset\C^2,$ as claimed. 
\end{proof}

\begin{rmrk} \label{rmrk:counter}
In Cartan's classification \cite[p.~70]{Ca1}, the left-invariant spherical elliptic CR structure on $\SLt$ appears in  item $5^\circ$(B) of the first table, as a left-invariant CR structure on the group $\rm{Aff}(\R)\times \R/\Z$. This group   is {\em not} isomorphic to $\SLt$, yet it admits a left-invariant spherical structure, CR equivalent to the spherical elliptic CR structure on $\SLt$. This shows that the asphericity condition in Theorem \ref{thm:ens}
is essential. Both groups are subgroups of the full 4-dimensional group of automorphism of this homogeneous spherical CR manifold (the stabilizer in $\PUto$ of a chain in $S^3$). The hyperbolic spherical structure is item $8^\circ$(K$^\prime$).

 The elliptic and hyperbolic aspherical left-invariant structures on $\SLt$ appear in  items $4^\circ$(K) and $5^\circ$(K$^\prime$) (respectively) of the second table. In these items, Cartan gives explicit equations for the adjoint orbits   in inhomogeneous coordinates $(x,y)\in\C^2\subset\CP^2$ (an affine chart). For the elliptic aspherical orbits  he gives the equation $1+x\bar x- y\bar y=\mu|1+x^2-y^2|,$ with $\Im(x(1+\bar y))>0$ and $\mu>1;$  for the hyperbolic aspherical structures he gives the equation $x\bar x+y\bar y-1=\mu|x^2+y^2-1|,$ with $(x,y)\in\C^2\setminus\R^2$ and $0<|\mu|<1.$ Both equations are  $\tr(L\bar L)=\mu|\tr(L^2)|,$ with $(x,y)=(b+c,b-c))/(2a)$ in the elliptic case, and $(x,y)=(2a, b-c)/(b+c)$ in he hyperbolic case. The elliptic orbits are the generic orbits in   the exterior of $S^3$, given by $\tr(L\bar L)>0$, while  the hyperbolic orbits lie in its interior, given by $\tr(L\bar L)<0$. The elliptic orbits come in complex-conjugate pairs; that is, for each orbit, given by the pairs of roots $\zeta_1,\zeta_2\in\C\setminus\R$ in the same (fixed) half-plane,  with a fixed hyperbolic distance $d(\zeta_1,\zeta_2)$, there is a complex-conjugate orbit where the  pair of roots lie in the opposite half plane. The condition  $\Im(x(1+\bar y))>0$ constrain the roots to lie in one of the half planes, so picks up  one of the orbits in each  conjugate pair. The hyperbolic orbits are self conjugate. 
 \end{rmrk}

\section{$\SUt$}\label{sec:su2}
$\SUt\simeq S^3$ is the group of $2\times 2$ complex unitary matrices with det=1. Its Lie algebra $\sut$ consists of anti-hermitian $2\times 2$ complex matrices with $\sut\otimes\C=\sltc$. This case is easier then the previous case of $\SLt$, with no really new ideas, so we will be much briefer. The outcome is that there is a single 1-parameter family of left-invariant CR structures, exactly one of which is  spherical, the standard spherical structure in $S^3$, realizable in $\C^2$. The rest of the structures are 4:1 covers of generic adjoint  orbits in $\Pgc\simeq\CP^2$.

\begin{lemma}
$\Aut(\SUt)=\Aut(\sut)=\Inn(\SUt)=\SUt/\{\pm \II\}\simeq \SO_3.$
\end{lemma}
\begin{proof}Similar to the $\SLt$ case, the Killing form and the triple product on $\sut$ are defined in terms 
of the Lie bracket alone. This gives a natural inclusion $ \Aut(\SUt)\subset \SOt$. 
The conjugation action gives an embedding $\Inn(\SUt)=\SUt/\{\pm \II\}\subset \SO_3$. 
The last two groups are connected and 3-dimensional, hence coincide.
\end{proof}

Since $\SUt\subset \SLtc$, with $(\sut)_\C=\sltc$, we can, like in the previous case of $G=\SLt$,  identify $\P((\sut)_\C),$ $\SUt$-equivariantly, with $S^2(\CP^1)$, the set of unordered pairs of points on $\CP^1=S^2$, with $\Aut(\SUt)=\SUt/\{\pm\II\}=\SO_3$ acting on $S^2(\CP^1)$ by euclidean rotations of  $\CP^1=S^2$, and complex conjugation in $\P((\sut)_\C)$ given by the antipodal map. Hence $\P((\sut)_\C)$ consists of non-antipodal unordered pairs of points $\zeta_1, \zeta_2\in S^2$, each of which is given uniquely, up to $\Aut(\SUt)=\SO_3$, by  their  spherical distance $d(\zeta_1, \zeta_2)\in [0,\pi).$

\begin{prop}\label{prop:sut}
Let $V_t\subset T_\C\SUt$, $t\in\R,$  be the left-invariant complex line bundle spanned at $e\in\SUt$ by  
\be\label{eq:LL}
L_t=\left(\begin{array}{cc}
0 &t-1\\
t+1& 0
\end{array}
\right)\in\sut\otimes\C=\sltc.\ee
 Then 
\benum
\item $V_t$ is a left-invariant CR structure on $\SUt$ for all $t\neq 0.$
\item $V_t$ is spherical if and only if $t=\pm 1$. 
\item Every  left-invariant  CR structure on $\SUt$  is CR equivalent to 
 $V_t$ for  a unique  $t\geq 1.$
\item The  aspherical left-invariant CR structures $V_t$, $t>1$, are  pairwise non-equivalent, even locally.  
\item $V_1$ is realized by any of the non-null orbits of the standard representation  of $\SUt$ in $\C^2$. The aspherical structures are locally realized as  $4:1$ covers of the adjoint orbits of $\SUt$ in $\Psltc$. 
\end{enumerate}
\end{prop}

\begin{proof}
(a) Note that $L_t\in\sut$ only for $t=0$ and that $\sut$  does not have 2-dimensional subalgebras. It follows that  $[L_t]$ is regular  for all $t\neq 0$.

\sn (b)  We apply  Proposition \ref{prop:CRc}. The left-invariant $\sut$-valued Maurer Cartan form  on $\SUt$ is 
\be\label{eqn:mc}\Theta=g^{-1}\d g= \left(\begin{array}{cc}
i\alpha &\beta+i\gamma\\
-\beta+i\gamma & -i\alpha 
\end{array}.
\right)
\ee
 The Maurer Cartan equation $\d\Theta=-\Theta\wedge\Theta$ gives
 $$
 \d\alpha=-2\beta\wedge\gamma, \  \d\beta=-2\gamma\wedge\alpha, \  \d\gamma=-2\alpha\wedge\beta.
 $$ 
A coframe well adapted to $V_t$ is 
$$\phi=\alpha,\ \phi_1=\sqrt{t}\beta+{i\over\sqrt{t}}\gamma,
$$
satisfying 
$$\d\phi=i\phi_1\wedge\bar\phi_1,\quad \d\phi_1=-i\left( {1\over t}+t \right)\phi\wedge\phi_1 -i\left( {1\over t}-t \right)\phi\wedge\bar\phi_1. 
$$
We conclude from  Proposition \ref{prop:CRc} that   $V_t$ is spherical if and only if 
 $\left( {1\over t}+t \right)\left( {1\over t}-t \right)=0;$ that is, $t=\pm 1.$
 
 \sn(c)  The quadratic polynomial associated to $L_t$ is $(t+1)\zeta^2-(t-1)$, with roots $\zeta_\pm=\pm\sqrt{(t-1)/(t+1)}$. For $t=1$ (the spherical structure) this is a double point at $\zeta=0$, and for $t>1$ these are a pair of points  symmetrically situated on the real axis, in the interval $(-1,1)$. As $t$ varies from $1$ to $\infty$ the spherical distance $d(\zeta_+, \zeta_-)$ increases monotonically from $0$ to $\pi$ (see next paragraph). It follows that every pair of unordered non-antipodal pair of points on $S^2$ can be mapped by $\Aut(\SUt)=\SO_3$ to a  pair $\zeta_\pm$ for a unique $t\geq 1$. 
 
 One way to see the claimed statement about $d(\zeta_+, \zeta_-)$ is to place the roots  on the sphere $S^2$, using the inverse stereographic projection $\zeta\mapsto   (2\zeta, 1-|\zeta|^2)/(1+|\zeta|^2)\in\C\oplus\R$. Then $\zeta_\pm\mapsto (\pm \sin\theta,0,\cos\theta)\in\R^3$, where $\cos\theta=1/t$ and $\theta\in[0,\pi/2)$ for $t\in[1,\infty).$ Thus as $t$ increases from $t=1$ to $\infty$ the pair of  points on $S^2$ start  from a double point at $(1,0,0)$, move in opposite directions along the meridian $y=0$ and  tend towards the poles $ (0,0,\pm 1)$ as $t\to\infty$.

 \sn (e) Every non-null orbit of the standard action of $\SUt$ on $\C^2$ contains a point of the form $v=(\lambda,0),$ $\lambda\in \C^*$. Since the stabilizer of such a point is trivial and $L_1v=0$, it follows by  Lemma \ref{lemma:real} that $g\mapsto gv$ is a CR embedding of $V_1$ in $\C^2$.  
 For $t>1$,  we use Proposition \ref{prop:real} to realize the aspherical CR structure $V_t$ as the $\SUt$-orbit of $[L_t]$ in $\Psltc$. The stabilizer in $\SO_3$ is the two element group interchanging the two roots in $S^2$, hence the stabilizer in $\SUt$ is a 4 element subgroup. 
\end{proof}

\begin{rmrk}As in the $\SLt$ case (see Remark \ref{rmrk:short}),  there is a somewhat quicker way to prove  item (c). First note that $\Aut(\SUt)=\SO_3$ acts transitively on the set of  2-dimensional subspaces of $\sut$, hence one can fix the contact plane $D_e$ arbitrarily, say $D_e=\Ker(\alpha)=\Span\{B,C\},$  where $A,B,C$ is the basis of $\sut$ dual to $\alpha, \beta, \gamma$ of equation \eqref{eqn:mc}. Then,  using  the    subgroup $\rm{O}_2\subset \SO_3=\Aut(\SUt)$ leaving  invariant $D_e$, one can map any almost complex structure on $D_e$ to $J_t:B\mapsto tC$, for a unique $t\geq 1,$ with  associated $(0,1)$-vector  $B+itC=-L_t$. 
\end{rmrk}

\begin{rmrk}\label{rmrk:burns} Proposition \ref{prop:sut}(e) gives a $4:1$  CR immersion  $\SUt\to \Psltc\simeq\CP^2$  of each of the aspherical left-invariant CR structures $V_t$, $t>1$. In fact, the proof shows that   $\SUt\to\sltc\simeq\C^3$, $g\mapsto gL_tg^{-1}$, is a $2:1$ CR-immersion. It is still unknown, as far as we know, if one can find immersions into $\C^2$. However, it is known that one cannot find  CR {\em embeddings} of the aspherical $V_t$ into $\C^n$, $n\geq 2$. This was first proved in  \cite{Ro}, by showing that any function $f:\SUt\to \C$ which is CR with respect to any of the  aspherical $V_t$
 is necessarily {\em even}, i.e.   $f(-g)=f(g).$ A simpler representation theoretic argument was later  given in \cite{Bu}, which we proceed to  sketch here (with minor notational modifications). 
 
 First, one embeds    $\mu:\SUt\to \C^2$, $g\mapsto g{1\choose 0},$ with image $\mu(\SUt)=S^3$, mapping the action of $\SUt$ on itself by left translations  to  the restriction to $S^3$ of the standard linear action  of $\SUt$ on $\C^2$. Next, one uses the  `spherical harmonics' decomposition  $L^2(S^3)=\bigoplus_{p,q\geq 0} H^{p,q}$, where $H^{p,q}$ is the  restriction to $S^3$ of the complex homogenous  harmonic polynomials on $\C^2$ of bidegree $(p,q)$; that is, complex polynomials $f(z_1, z_2, \bar z_1, \bar z_2)$ which are homogenous of degree $p$ in $z_1,z_2$,  homogenous of degree $q$ in $\bar z_1, \bar z_2$, and satisfy $(\partial_{z_1}\partial_{\bar z_1}+\partial_{z_2}\partial_{\bar z_2}) f=0$. Each $H^{p,q}$ has dimension $p+q+1$, is $\SUt$-invariant and irreducible, with $-\II\in\SUt$ acting by $(-1)^{p+q}.$ 
 
 Next, one checks that  $Z:=\bar z_2\partial_{z_1}-\bar z_1\partial_{z_2}$  is an $\SU_2$-invariant $(1,0)$-complex vector field on $\C^2$,  tangent to $S^3$,  mapping  $H^{p,q}\to H^{p-1,q+1}$ for all  $p>0, q\geq 0$, $\SUt$-equivariantly. The latter  is a non-zero map, hence, by Schur's Lemma, it is an {\em isomorphism}.  Similarly,  $\bar Z$ is a $(0,1)$-complex vector field on $\C^2$, tangent to $S^3$, defining   an $\SUt$-isomorphism  $H^{p,q}\to H^{p+1,q-1}$  for all  $q>0, p\geq 0$. It follows that each $H^k:=\bigoplus_{p+q=k}H^{p,q}$, $k\geq 0$,  is invariant under $Z, \bar Z.$ 
 
 Next, one checks that $\bar Z_t:=(1+t)\bar Z+ (1-t)Z$, restricted to $S^3$,  spans $\d\mu(V_t)$. That is,  $f:S^3\to \C$ is CR with respect to $\d\mu(V_t)$ if and only if $\bar Z_t f=0$. By the previous paragraph, each $H^k$ is $\bar Z_t $ invariant, hence  $\bar Z_t f=0$ implies $\bar Z_t f^k=0$ for all $k\geq 0$, where $f^k\in H^k$ and $f=\sum f^k$. Now  one uses the previous paragraph to show  that for $k$ odd and $t>1$, $\bar Z_t$ restricted to $H^k$ is {\em invertible}. It follows that $\bar Z_t f=0$, for $t>1$, implies  that $ f^k=0$ for all $k$ odd;  that is, $f$ is even, as claimed. \qed

\end{rmrk}

\begin{rmrk} In Cartan's classification \cite[p.~70]{Ca1}, the spherical structure $V_1$ is item $1^\circ$ of the first table. The  aspherical  structures appear in  item $6^\circ$(L) of the second table. Note  that Cartan has an error  in this item (probably typographical): the equation for the $\SUt$-adjoint orbits, in homogenous coordinates in $\CP^2$,  should be $x_1\bar x_1+x_2\bar x_2+x_1\bar x_2=\mu|x_1 ^2+x_2^2+x_3^2|,$ $\mu>1$ (as appears correctly on top of p.~67). This is a coordinate version of the equation $\tr(L\bar L^t)=\mu|\tr(L^2)|$.
\end{rmrk}

\section {The Heisenberg group}

The Heisenberg group $H$  is the group of matrices of the form 
$$\left(\begin{array}{ccc}1&x&z\\0&1&y\\ 0&0&1\end{array}\right), \quad x,y,z\in\R.$$
Its Lie algebra  $\h$  consists of matrices of the form 
$$\left(\begin{array}{ccc}0&a&c\\0&0&b\\ 0&0&0\end{array}\right), \quad a,b,c\in\R.$$

\begin{lemma} $\Aut(H)=\Aut(\h)$ is the 6-dimensional Lie group, acting on $\h$ by 
\be\label{eq:auth}
\left(\begin{array}{cc}T&0\\ \vb &\det(T)\end{array}\right), \quad T\in\GL_2(\R),\ \vb\in\R^2
\ee
(matrices with respect to the basis dual to $a,b,c$). 
\end{lemma}
\begin{proof}Let $A,B,C$ be the basis of $\h$ dual to $a,b,c$. Then $$[A,B]=C,\ [A,C]=[B,C]=0.$$
One can then verify by a direct calculation that the matrices in formula \eqref{eq:auth} are those preserving these commutation relations.
\end{proof}

\begin{rmrk} Here is a cleaner proof of the last Lemma (which works also for the higher dimensional Heisenberg group):  the commutation relations imply that   $\mathfrak{z}:=\R C$ is the center of $\h$, so any  $\phi\in\Aut(H)$ leaves it invariant, acting on $\mathfrak{z}$ by some $\lambda\in\R^*$ and on $\h/\mathfrak{z}$ by some $T\in\Aut(\h/\mathfrak{z}).$ The Lie bracket defines a non-zero element $\omega\in \Lambda^2((\h/\mathfrak{z})^*)\otimes \mathfrak{z}$ fixed by $\phi$. Now  $\phi^*\omega=(\lambda/\det(T))\omega$, hence $\lambda=\det(T)$. This gives the desired form of $\phi$, as in equation \eqref{eq:auth}.
\end{rmrk}

\begin{prop} Let $V\subset T_\C H$   be the left-invariant complex line bundle spanned at $e\in H$ by  
\be\label{eq:H}
L=\left(\begin{array}{ccc}
0 &1&0\\
0&0& i\\
0&0&0
\end{array}
\right)\in\h\otimes\C.\ee
 Then 
\benum
\item $V$ is the unique  left-invariant CR structure on $H$, up to the action of $\Aut(H)$. 
\item V is spherical, CR equivalent to  the complement of a point in $S^3$. 
\item $V$ is  also embeddable in $\C^2$ as the real quadric $
\Im(z_1)=|z_2|^2.$ In these coordinates, the group multiplication in $H$ is given by 
\[
(z_1,z_2)\cdot (w_1,w_2)=(z_1+w_1,z_2+w_2+2iz_1\bar{w_1}).
\]

\end{enumerate}
\end{prop}
\begin{proof}
(a) The adjoint action is $(x,y,z)\cdot(a,b,c)=(a, b, c + b x - a y).$ This has  1-dimensional  orbits, the affine lines parallel to the $c$ axis, except the $c$ axis itself (the center of $\h$), which is pointwise fixed. 
The `vertical' 2-dimensional  subspaces in $\h$, i.e.  those containing the $c$ axis, are subalgebras, so give degenerate CR structures. It is easy to see that any other 2-dimensional subspace can be mapped by the adjoint action  to $D_e=\{c=0\}$ and that the subgroup of $\Aut(H)$ preserving $D_e$  consists of   
$$\left(\begin{array}{cc}T&0\\0&\det(T)\end{array}\right), \quad T\in\GL_2(\R),$$
(written with respect to the basis of $\h$ dual to $a,b,c$). 
These act transitively on the set of almost complex structures on $D_e$. One can thus take the almost complex structure on $D_e$ mapping $A\mapsto B,$ with associated $(0,1)$ vector $L=A+iB.$  

\sn (b) Define a Lie algebra homomorphism  $\rho':\h\to \End(\C^3)$
\be \label{eq:rep}
(a,b,c)\mapsto
\left(\begin{array}{ccc}0&-b-ia&2c\\ 0&0&a+ib\\ 0&0&0\end{array}\right).\ee
with associated complex linear representation $\rho:H\to \GL_3(\C)$, 
\be \label{eq:repg}
(x,y,z)\mapsto
\left(\begin{array}{ccc}1&-y-ix&2z-xy-{i\over 2}(x^2+y^2)\\ 0&1&x+iy\\ 0&0&1\end{array}\right).
\ee
Then one can verify that $\rho$ has  the following properties: 
\bitem
\item  It preserves  the pseudo-hermitian quadratic form $ |Z_2|^2-2\Im(Z_1\bar Z_3)$  on $\C^3$, of signature $(2,1)$.
\item The induced $H$-action on $S^3\subset \CP^2$ (the projectivized  null cone of the pseudo-hermitian form) has 2 orbits:  a fixed point  $[\e_1]\in S^3$ and its complement. 
\item The $H$-action on  $S^3\setminus\{[\e_1]\}$ is free.
\item $\rho'(L)\e_3=0.$
\end{itemize}
It follows, by Lemma \ref{lemma:real}, that $H\to S^3\subset\CP^2,$ $h\mapsto [\rho(h)\e_3]$, is a  CR embedding of the CR structure $V$ on $H$ in $S^3$,  whose image is the complement of $[\e_1]$.   

\sn (c) In the  affine chart $\C^2\subset \CP^2$, $(z_1, z_2)\mapsto [z_1:z_2:1]$, the equation of $H=S^3\setminus [\e_1]$ is  $2\Im(z_1)=|z_2|^2.$ After  rescaling the $z_1$ coordinate one obtains $\Im(z_1)=|z_2|^2.$ The claimed formula for the group product in these coordinates follows from  the embedding $h\mapsto [\rho(h)\e_3]$ and formula \eqref{eq:repg}. \end{proof}

\begin{rmrk} The origin of  formula \eqref{eq:rep} is as follows. Consider the standard representation of  $\SUto$ on $\C^{2,1}$ and the resulting action on $S^3\subset\CP^2=\P(\C^{2,1})$. The stabilizer in $\SUto$ of a point  $\infty \in S^3$ is a 5-dimensional subgroup $P\subset\SUto$, acting transitively on $S^3\setminus\{\infty\}.$ The stabilizer in $P$ of a point $o\in S^3\setminus\{\infty\}$ is a subgroup $\C^*\subset P$, whose conjugation action on $P$ leaves invariant a 3-dimensional normal subgroup of $P$, isomorphic to our $H$, so that $P= H\rtimes\C^*$. To get formula \eqref{eq:rep}, we   consider the adjoint action of $\C^*$ on the Lie algebra $\mathfrak{p}$ of $P$, under which  $\mathfrak{p}$ decomposes  as  $\mathfrak{p}=\h\oplus \C$,  as in  \eqref{eq:rep}.   For more details, 
see \cite[pp.~115-120]{Gol2}. \end{rmrk}

\begin{rmrk} In Cartan's classification \cite[p.~70]{Ca1}, the left-invariant spherical structure on $H$ is item $2^\circ$(A) of the first table. 
\end{rmrk} 

\section{The Euclidean Group}

Let $\E=\SO_2\rtimes\R^2$ be the group of orientation preserving isometries of $\R^2$, equipped with the standard euclidean metric. Every element  in $\E$ is of the form $\vb\mapsto R\vb+\w$, for some $R\in \SO_2$, $\w\in\R^2$. 
If we embed $\R^2$ as the  affine plane  $z=1$ in $\R^3$, $\vb\mapsto (\vb,1)$, then $\E$ is identified with the subgroup of $\GL_3(\R)$ consisting of matrices in block  form 
\be\label{eq:E} \left(\begin{array}{ccc}R&\w\\  0&1\end{array}\right), \quad R\in\SO_2, \ \w\in\R^2.
\ee
Its Lie algebra  $\Ee_2$  consists of matrices of the form 
\be\label{eq:Ee}\left(\begin{array}{ccc}0&-c&a\\ c&0&b\\ 0&0&0\end{array}\right), \quad a,b,c\in\R.
\ee

Let $\CE_2$ be the group of {\em similarity} transformations of $\R^2$ (not necessarily orientation preserving). That is, maps $\R^2\to\R^2$  of the form $\vb\mapsto T\vb+\w$, where $\w\in\R^2$,  $T\in \mathrm{CO}_2=\R^*\times \rm{O}_2$.
Then  $\E\subset \CE_2$ is a normal subgroup with trivial centralizer,  hence there is a natural inclusion $\CE_2\subset \Aut(\E)$. 
\begin{lemma}
$\CE_2= \Aut(\E)=\Aut(\Ee_2)$. 
\end{lemma}

\begin{proof}
One calculates that the inclusion $\CE_2\subset  \Aut(\Ee_2)$ is given, with respect to the basis $A,B,C$ of $\Ee_2$ dual to $a,b,c$, by the matrices 
\be\label{eq:aut}
(\w,T)\mapsto\left(\begin{array}{cc}
T&-\epsilon i\w\\
0&\epsilon
\end{array}
\right),  \quad  T\in \rm{CO}_2, \ \w\in\R^2, 
\ee
where $\epsilon=\pm 1$ is the sign of $\det(T)$ and $i:(a,b)\mapsto (-b,a).$ To show that the map $\CE_2\to \Aut(\Ee_2)$ of equation \eqref{eq:aut}  is surjective, let $\phi\in\Aut(\Ee_2)$ and observe that $\phi$  must preserve the subspace $c=0$, since it is the unique 2-dimensional ideal of $\Ee_2$. Thus 
$\phi$ has the form
$$\phi=\left(\begin{array}{rrr}a_{11}&a_{12}&a_{13}\\ a_{21}&a_{22}&a_{23}\\ 0&0&a_{33}\end{array}\right)$$
with respect to the basis $A,B,C$ of $\Ee_2$ dual to $a,b,c$. 
Next, using the commutation relations 
\begin{equation}\label{eq:com}
[A,B]=0, \ [A,C]=-B, \ [B,C]=A.
\ee 
 we get
$$a_{11}=a_{22}a_{33}, \ a_{22}=a_{11}a_{33},\ a_{12}=-a_{21}a_{33},\ a_{21}=-a_{12}a_{33}.
$$
From the first two equations we get $a_{11}=a_{11}(a_{33})^2, $ and from the last two
 $a_{12}=a_{12}(a_{33})^2.$ We cannot have $a_{11}=a_{12}=0$, else $\det(\phi)=
 (a_{11}a_{22}-a_{12}a_{21})a_{33}=0.$ It follows that $a_{33}=\pm 1$. If $a_{33}=1$ then we get from the above 4 equations $a_{22}=a_{11}, a_{12}=-a_{21}$, hence 
 the top left $2\times 2 $ block of $\phi$ is in $\mathrm{CO}_2^+$  (an orientation preserving linear similarity). If $a_{33}=-1$ then we get  $a_{22}=-a_{11}, a_{12}=a_{21}$, hence the top left $2\times 2 $ block of $\phi$ is in $\mathrm{CO}_2^-$  (an orientation reversing linear similarity). These are exactly the matrices of equation \eqref{eq:aut}. 
 \end{proof}

\begin{prop}Let $V\subset T_\C \E$ be the left-invariant line bundle whose value at $e\in \E$ is spanned by 
$$L=\left(\begin{array}{ccc}0&-i&1\\ i&0&0\\ 0&0&0\end{array}\right)\in (\Ee_2)_\C.$$
Then 

\benumm
\item Every   left-invariant CR structure on $\E$ is CR equivalent to $V$ by $\Aut(\E)$. 
\item $V$ is an aspherical left-invariant CR structure on $\E$. 
\item $V$ is  realized in $\P((\Ee_2)_\C)=\CP^2$ by the adjoint orbit of $[L]$. This is CR equivalent to  the real hypersurface $[\Re(z_1)]^2+[\Re(z_2)]^2=1$ in $\C^2$. 

\end{enumerate}
\end{prop}

\begin{proof}
(a) Let $A,B,C$ the basis of $\Ee_2$ dual to $a,b,c$. Then $L=A+iC$, so  $D_e=\Span\{A,C\}=\{b=0\}.$ The plane $c=0$  is a subalgebra of $\Ee_2$, so gives  a degenerate CR structure. By equation \eqref{eq:aut}, 
every other plane can be mapped by  $\Aut(\E)$ to $D_e$. The subgroup of $\Aut(\E)$ preserving $D_e$ acts on $D_e$, with respect to the basis $A,C$, by the  matrices 
$$\left(\begin{array}{cc}r &s\\  0&\epsilon\end{array}\right), \quad r\in\R^*, \ s\in \R, \ \epsilon=\pm1.$$
 One can then  show that this group acts transitively on the space  of almost complex structures on $D_e$.

\sn (b)   Let $\alpha, \beta, \gamma$ be the  left-invariant 1-forms on $E$ whose value at $e$ is $a,b,c$ (respectively).  Then
$$\Theta=\left(\begin{array}{ccc}0&-\gamma&\alpha\\ \gamma&0&\beta\\ 0&0&0\end{array}\right)$$
is the left-invariant Maurer-Cartan form on $E$, satisfying $\d\Theta=-\Theta\wedge\Theta$,  from which we get 
\be\label{eq:stre}
\d \alpha=-\beta\wedge\gamma, \  \d\beta=\alpha\wedge\gamma, \ \d\gamma=0. 
\ee 

A coframe $(\phi, \phi_1)$ adapted to $V$ (i.e.  $\phi(L)=\phi_1(L)=0,$ $\bar\phi_1(L)\neq 0$) is
$$\phi=\beta, \ \phi_1={1\over\sqrt{2}}\left(\alpha+i\gamma\right).$$ 
Using  equations \eqref{eq:stre}, we find 
$$\d\phi=i\phi_1\wedge\bar\phi_1, \quad \d\phi_1={i\over 2}\phi\wedge\phi_1-{i\over 2}\phi\wedge\bar\phi_1,$$
Thus $(\phi, \phi_1)$ is well-adapted. By Proposition \ref{prop:CRc}, the structure is aspherical. 

\sn(c) Using Proposition \ref{prop:real}, this amount to showing that the stabilizer of $[L]$ in $\E$ is trivial. This is a simple calculation using formula \eqref{eq:aut}, with $L=A+iC$  and $T\in\SO_2,$ $\epsilon=1$.  The $\E$-orbit of $[L]$ in $\P((\Ee_2)_\C)$ is contained in the affine chart $c\neq 0$. Using  the coordinates $z_1=a/c, z_2=b/c$ in this chart, the equation for the orbit  is $[\Re(z_1)]^2+[\Re(z_2)]^2=1.$ 
\end{proof}

\begin{rmrk} In Cartan's classification \cite[p.~70]{Ca1}, the left-invariant aspherical structure on $E_2$ is item $3^\circ$(H) of the second table, with $m=0$.  
\end{rmrk} 

\appendix

\section{The Cartan equivalence method}

We state the main result of  \'E.~Cartan's method of equivalence, as implemented for  CR geometry in  \cite{Ca1},   and apply it to  left-invariant CR structures on Lie groups.  We follow mostly the 
 notation and terminology of \cite{Ja}.  
  
The equivalence method associates canonically  to  each CR 3-manifold $M$ an $H$-principal bundle $B\to M$, where $H\subset \PUto=\SUto/\Z_3$ is the stabilizer of a point in $S^3\subset \CP^2=\P(\C^{2,1})$ (a 5-dimensional parabolic subgroup). Furthermore, $B$ is  equipped  with a certain 1-form  $\Theta:TB\to \suto$, called the  {\em Cartan connection form}, whose eight components are linearly independent at each point, defining a coframing on $B$  (an `$e$-structure').
In the special case of $M=S^3$, equipped with its standard spherical structure, $B$ can be identified with $\PUto$ and $\Theta$ with  the left-invariant Maurer-Cartan form on this group. The {\em curvature} of $\Theta$ is the $\suto$-valued 2-form $\Omega:=\d\Theta+\Theta\wedge\Theta$. It vanishes if and only if $M$ is spherical and  is the basic local invariant of  CR geometry, much like the Riemann curvature tensor in Riemannian geometry. The construction is canonical  in the sense that each  CR equivalence $f:M\to M'$ lifts uniquely to a bundle map  $\tilde f:B\to B'$, preserving the coframing, i.e.  $\tilde f^*\Theta'=\Theta.$ In fact, $B$ is an $H$-reduction  of the second order frame bundle of $M$ (the 2-jets of germs of local diffeomorphisms $(\R^3, 0)\to M$), and  $\tilde f$ is the restriction of the 2-jet of $f$ to $B$.

More concretely, fix a pseudo-hermitian form  on $\C^3$ of signature  $(2,1)$, 
 $(z_1,z_2,z_3)\mapsto |z_2|^2+i(z_3\bar z_1-z_1\bar z_3)$, and let $\SUto\subset \SL_3(\C)$ be the subgroup preserving this hermitian form. A short calculation shows that its Lie algebra $\suto$ consists of matrices of the form 
\be\label{eq:suto}
\left(\begin{array}{ccc}
{1\over 3}(\bar c_2+2c_2)&i\bar c_3&-c_4\\
c_1&{1\over 3}(\bar c_2-c_2)&-c_3\\
c&i\bar c_1&-{1\over 3}(c_2+2\bar c_2)
\end{array}
\right),
\ee
 where $c,c_4\in\R$  and $c_1,c_2,c_3\in\C$.  Accordingly,  $\Theta$ decomposes as 
\be\label{eq:cc}
\Theta=
\left(\begin{array}{ccc}
{1\over 3}(\bar \theta_2+2\theta_2)&i\bar \theta_3&-\theta_4\\
\theta_1&{1\over 3}(\bar \theta_2-\theta_2)&-\theta_3\\
\theta&i\bar \theta_1&-{1\over 3}(\theta_2+2\bar \theta_2)
\end{array}
\right),
\ee
where $\theta,\theta_4$ are real-valued  and $\theta_1,\theta_2,\theta_3$ are complex-valued 1-forms on $B$. Let
 $H\subset \PUto$ be the stabilizer of $[1:0:0]\in S^3\subset \CP^2$. Its  Lie algebra $\h\subset\suto$ is given by setting  $c=c_1=0$ in formula \eqref{eq:suto}. 
  In the case of the spherical CR structure on  $S^3$,  where $\Theta$ is the  left-invariant  Maurer-Cartan form on $B=\PUto$,  the Maurer-Cartan equations give
$\Omega=\d\Theta+\Theta\wedge\Theta=0.$ In general, $\Omega$ does not vanish but has  a rather special form.

We  summarize  Cartan's main result of \cite{Ca1}, as presented in \cite{Ja}. We first give a global version, then a local one, using adapted  coframes. Each has its  advantage. 
\begin{thm}[Cartan's equivalence method, global version] \label{thm:glob}
With each CR 3-manifold  $M$  there is  canonically associated an $H$-principal bundle  $B\to M$ with   Cartan connection $\Theta:TB\to \suto$, satisfying  
\benum
\item ($H$-equivariance)  $R_h^*\Theta=\Ad_{h^{-1}}\Theta$ for all $h\in H$. 

\item  The vertical distribution on $B$ (the tangent spaces to the fibers of $B\to M$) is given by $\theta=\theta_1=0.$

\item ($e$-structure) The eight  components of $\Theta$, namely 
$\theta,$ $ \Re(\theta_1),$ $\Im(\theta_1),$ $ \Re(\theta_2),$ $\Im(\theta_2),$ $\Re(\theta_3),\Im(\theta_3),$ $ \theta_4$,   are pointwise linearly independent, defining  a coframing on $B$.

\item (The CR structure equations) There exist functions $ R,S:B\to\C$ such that 
$$\Omega=\d\Theta+\Theta\wedge\Theta=\left(\begin{array}{ccc}
0&-i \bar R &S\\
0&0&0\\
0&0&0
\end{array}
\right)\theta\wedge\theta_1
+
\left(\begin{array}{ccc}
0&0&\bar S\\
0&0&R\\
0&0&0
\end{array}
\right)\theta\wedge\bar\theta_1.
$$
Explicitly, 
\begin{align}\label{eq:cr1}
\begin{split}
\d\theta&=i\theta_1 \wedge\bar \theta_1-\theta \wedge(\theta_2+\bar\theta_2),  \\
\d\theta_1&=-\theta_1 \wedge\theta_2-\theta \wedge\theta_3,  \\
\d\theta_2&= 2i\,\theta_1\wedge\bar\theta_3+i\,\bar\theta_1\wedge\theta_3-\theta\wedge\theta_4,  \\
\d\theta_3&= -\theta_1\wedge\theta_4  -\bar\theta_2\wedge\theta_3-R\,\theta\wedge\bar\theta_1,\\
\d\theta_4&= i\,\theta_3\wedge\bar\theta_3-(\theta_2+\bar\theta_2)\theta_4+(S\,\theta_1+\bar S\,\bar\theta_1)\wedge\theta. 
\end{split}
\end{align}

\item (Spherical structures) $M$  is spherical if and only if $R\equiv 0$, in which case $S\equiv 0$ as well, hence  $\Omega\equiv 0.$ 

\item (Aspherical structures) If $M$ is aspherical, i.e.  $R$ is non-vanishing, then  $B_1=\{R=1\}\subset B$ is a $\Z_2$-principal  subbundle of $B$. The restriction of $(\theta_, \theta_1)$ to $B_1$ defines a coframing on it. 
 
 \item Any local CR diffeomorphism of CR manifolds $f:M\to M'$ lifts uniquely to an $H$-bundle map $\tilde f:B\to B'$ with $\tilde f^*\Theta'=\Theta.$ 
\end{enumerate}
\end{thm}

Here is a reformulation of the last theorem using {\em adapted coframes}. Note that such  coframes always exists, locally, for any CR manifold. See Definition \ref{def:adap} and the paragraph following it. 
\begin{thm}[Cartan's equivalence method, local version]
Let $M$ be a CR 3-manifold with an adapted coframe $(\phi, \phi_1)$, satisfying $\d \phi=i\phi_1\wedge\bar\phi_1 \  (\mod \phi)$.  Then  
\benumm
\item There exist on $M$ unique 
complex 1-forms $\phi_2,\phi_3$,  a real 1-form $\phi_4$ and complex functions $r,s$ such that 
\begin{align}\label{eq:cr2}
\begin{split}
&\d\phi=i\phi_1 \wedge\bar \phi_1-\phi \wedge(\phi_2+\bar\phi_2), \\
&\d\phi_1=-\phi_1 \wedge\phi_2-\phi \wedge\phi_3,  \\
&\d\phi_2= 2i\,\phi_1\wedge\bar\phi_3+i\,\bar\phi_1\wedge\phi_3-\phi\wedge\phi_4,  \\
&\d\phi_3= -\phi_1\wedge\phi_4  -\bar\phi_2\wedge\phi_3-r\,\phi\wedge\bar\phi_1,\\
&\d\phi_4= i\,\phi_3\wedge\bar\phi_3+(s\,\phi_1+\bar s\,\bar\phi_1)\wedge\phi. 
\end{split}
\end{align}
\item If $ (\phi,\phi_1)$  is well-adapted, i.e.  $\d\phi=i\phi_1\wedge\bar\phi_1,$ then $\phi_2$ is imaginary, $\phi_2+\bar\phi_2=0.$

\sn \item $M$ is spherical if and only if $r\equiv 0$, in which case $s\equiv 0$ as well. 

\sn \item If $M$ is aspherical, i.e.  $r$ is non-vanishing,   then there exist  on $M$  exactly two   \wa    coframes $(\tilde \phi, \tilde\phi_1)$ for which  $r=1$ in equations \eqref{eq:cr2},  given by $\tilde\phi=|\lambda|^2\phi,$ $\tilde \phi_1=\lambda(\phi+\mu\phi_1)$,  where $\lambda, \mu$ are complex functions given as follows: let $L$ be the complex vector field of type $(0,1)$ defined by $\theta(L)=\theta_1(L)=0,$ $\bar\theta_1(L)=1$, then $\lambda=\pm(|r|^{-1/2}\bar r)^{1/2}, $ $ \mu=i\,L(u)/u$  and  $ u=|\lambda|^2=|r|^{1/2}.$ 

\sn \item The previous items are related to  Theorem \ref{thm:glob} as follows:  
there exists a unique  section $\sigma:M\to B$ such that $\phi=\sigma^*\theta$ and $\phi_1=\sigma^*\theta_1.$ Furthermore, $\phi_i=\sigma^*\theta_i, $ $ i=2,3,4$,  $r=R\circ\sigma$ and $s=S\circ\sigma.$ If $M$ is aspherical then  $B_1$ is trivialized by the two sections corresponding to the two \wa coframes of the previous item. 
\end{enumerate}
\end{thm}

Proofs of these theorems are found in Chap.~6 and Chap.~7 of \cite{Ja}. 
Note that the  function $r$ in equations \eqref{eq:cr2}, sometimes called  `the Cartan CR curvature', is a {\em relative invariant} of the CR structure: only its vanishing is independent of the coframe.  Put differently, due to the $H$-equivariance of $\Theta$, and hence of $\Omega$, the function $R:B\to \C$ of Theorem \ref{thm:glob} varies non-trivially along any of the fibers of $B\to M$, unless it vanishes along it.

\begin{cor}\label{cor:aut} For any connected CR 3-manifold, 
\benumm
\item   $\Aut _\CR(M)$  and $\aut_\CR(M)$  are   a Lie group and a  Lie algebra (respectively) of dimension at most 8. The maximum dimension 8 is obtained if and only if $M$ is spherical. 

\item If $M$ is aspherical then $\Aut_\CR(M)$ and $\aut_\CR(M)$
have  dimension at most 3. 
\item $\Aut_\CR(S^3)=\PUto$.
\item  
If $U$ and $V$ are open connected subsets of $S^3$ and $f:U\to  V$ is a CR diffeomorphism then $f$ is the restriction to $U$ of some element in $\PUto$.

\end{enumerate}
\end{cor}

\begin{proof} (a) The essential observation is that any local diffeomorphism of  coframed manifolds, preserving the  coframing, is determined, in each connected component of its domain, by its value at a single point in it. This is a consequence of the  uniqueness theorem of solutions to ODEs. It follows that the  group of symmetries of a coframed connected manifold embeds in the manifold itself. This implies, by Theorem \ref{thm:glob} above, item (g), that $\Aut_\CR(M)$ embeds in  $B$, which is 8-dimensional. The same argument applies to $\aut_\CR(M)$, by restricting to open connected subsets of $M$. If $\dim\Aut_\CR(M)=8$, then it acts with open orbits in $B$, hence $R$ is locally constant. In particular, $R$ must be constant along the fibers of $B\to M$. By the $H$-equivariance of $\Omega$ this can happen only if $R$ vanish, which implies that $M$ is spherical, by Theorem \ref{thm:glob}, item (e). 

\sn (b)  If $M$ is aspherical then $\tilde f$ leaves $B_1$ invariant, preserving the coframing on it given by $(\theta, \theta_1)$. Then, as in the previous item,  $\Aut_\CR(M)$ embeds in $B_1$,  hence it is of dimension at most $3=\dim(B_1)$. 

\sn (c) As mentioned above, for $M=S^3$, $B=\PUto$ and $\Theta$ is the left-invariant Maurer-Cartan form. For any $f\in \Aut_\CR(M)$, let $\tilde f(e)=g=ge\in B$. This coincides with the action of $g$ on $\PUto$ by left translations, hence $\tilde f=g$. 

\sn (d) This is  the `unique  extension property' of Proposition \ref{prop:ext}. 
\end{proof}
In general, given a well-adapted coframe $\phi, \phi_1$,  it is  not so  simple to solve equations \eqref{eq:cr2} to find the associated one-forms and the functions $r,s$.  Fortunately, for a left-invariant  CR structure on a Lie group,  one can pick a left-invariant well-adapted coframe and then  it is straightforward to  write down explicitly the solutions   in terms of  $\phi, \phi_1$ and their structure constants.  
\begin{prop}\label{prop:CRc:Appendix}
Let $M$ be a  manifold with a CR structure given by a well-adapted   coframe $\phi, \phi_1$ satisfying
\begin{align}\label{eq:str1}
\begin{split}
\d\phi&=i\phi_1 \wedge\bar\phi_1,\\
\d\phi_1&=a\,\phi_1\wedge\bar \phi_1+b\,\phi\wedge \phi_1+c\,\phi\wedge \bar\phi_1,
\end{split}
\end{align}
for  some complex constants $a,b,c$. Then these constants satisfy  
\be \label{eq:str2}
\bar a c=ab, \ 
b+\bar b=0,
\ee
and 
equations  \eqref{eq:cr2} are satisfied by $r,s,\phi_j=A_j\phi + B_j\phi_1+C_j\bar\phi_1,\ j=2,3,4,$   
given by 
\begin{align*}
& A_2={i|a|^2\over 2}  +{3b\over 4},\
 B_2=\bar a ,\
 C_2=-a,\\
 &A_3={4ia b\over 3}, \
 B_3=\frac{i|a |^2}{2} -{b\over 4},\
 C_3=-c, \\
& A_4={\left| a\right| ^4\over 4} +{1\over 16} |b|^2+\frac{19}{12}  ib | a|^2- | c|^2,\
 B_4=\frac{2 \bar a b}{3} ,\
 C_4=\frac{2  a \bar b}{3}\\
 &  r=ic\left({|a|^2\over 3}+{3ib\over 2}\right),\
 s=\bar a \left(3|b|^2+{2i\over 3} |a|^2 b\right).
\end{align*}
\end{prop}
\begin{proof}
Taking exterior derivatives of  equations  \eqref{eq:str1} and substituting again  equations \eqref{eq:str1} in the result, we obtain 
equations  \eqref{eq:str2}.
The condition that $\phi_2$ is imaginary and $\phi_4$ is real is equivalent to 
$ A_2 =-\overline A_2, C_2=-\overline B_2,  A_4=\overline A_4, C_4=\overline B_4.$
Using this, substituting   $\phi_2,\phi_3,\phi_4$ into equations  \eqref{eq:cr2} and equating coefficients with respect to $\phi_1 \wedge\bar\phi_1$,\  $\phi\wedge \phi_1,\ \phi\wedge \bar\phi_1$ it is straightforward to obtain a system of algebraic equations whose solution is given  by the stated formulas (we used Mathematica). 
\end{proof}
\begin{cor}\label{cor:sph-curv}
A locally homogeneous CR structure given by an adapted coframe satisfying equation \eqref{eq:str1} is  spherical if and only if 
$
c(2\left| {a}\right | ^2+9ib)=0.
$
\end{cor}


\begin{thebibliography}{9}

\bibitem{A} H.~Alexander, {\em Holomorphic mappings from the ball and polydisc},
Math, Ann.
  {\bf  209.3} (1974), 249--256.
  
\bibitem{BE}D.~Burns, C.L.~Epstein, 
  {\em A global invariant for three dimensional {CR}-manifolds},
  Inven. math.
  {\bf  92.2} (1988), 333--348.
  
\bibitem{Bu}D.~Burns,   {\em Global behavior of some tangential Cauchy-Riemann equations}. In:   Partial differential equations and geometry,  Proc. Conf. Park City, Utah, 1977. Lect. Notes Pure Appl. Math. 48,
51--56. New York: Dekker

\bibitem{Cap} A.~{\v{C}ap}, {\em On left invariant {CR} structures on {SU(2)}},
  Arch. Math.  (Brno)
  {\bf  42} (supplement) (2006), 185--195.



\bibitem{Ca1}\'E.~Cartan, 
  {\em Sur la g\'eom\'etrie pseudo-conforme des hypersurfaces de deux variables complexes {I}},
  Ann. Math. Pura Appl.
  {\bf  11.4}  (1932), 17--90.

\bibitem{Ca2}\'E.~Cartan, 
  {\em Sur la g\'eom\'etrie pseudo-conforme des hypersurfaces de deux variables complexes {II}},
  Annali della Scuola Normale Superiore di Pisa, Classe di Scienze 2e s\'erie
  {\bf  1.4}  (1932), 333--354.


\bibitem{CM} S.S.~Chern, J.~Moser, 
  {\em Real hypersurfaces in complex manifolds},
  Acta mathematica
  {\bf  133.1}  (1974), 219--271
  
\bibitem{CaMo} A.~Castro, R.~Montgomery,
  {\em The chains of left-invariant Cauchy--Riemann structures on {SU(2)}},
    Pac. J.  Math.
  {\bf  238.1} (2008), 41--71.
  
  \bibitem{Cl} 
J.N.~Cleland, {\em
From Frenet to Cartan: The Method of Moving Frames}, 
Graduate Studies in Mathematics 178, American Mathematical Society, 2017.

  
\bibitem{ENS}F.~Ehlers, W.D.~Neumann, J.~Scherk, 
  {\em Links of surface singularities and {CR} space forms},
  Comm. Math. Helv.
  {\bf  62.1} (1987), 240--264.
  
\bibitem{Gol1}W.M.~Goldman, 
  {\em Locally homogeneous geometric manifolds}. In: Proceedings of the International Congress of Mathematicians 2010, Vol. II, 717--744. World Scientific, 2011. 

\bibitem{Gol2}W.M.~Goldman, 
{\em Complex hyperbolic geometry}. Oxford University Press (1999). 

\bibitem{Gr}
 P.~Griffiths,
{\em
 On Cartan's method of Lie groups and moving frames as applied to uniqueness and existence questions in differential geometry}, Duke Math. J.  {\bf 41} (1974), 775-814.
 
\bibitem{Ho}
L.~Hormander, 
 {\em
The analysis of linear partial differential operators. I. Distribution theory and Fourier analysis}, 
Grundlehren der Mathematischen Wissenschaften [Fundamental Principles of Mathematical Sciences], 256. Springer-Verlag, Berlin, 1983.

\bibitem{Ja} H.~Jacobowitz,   {\em An introduction to CR structures}. No.   32.   American Mathematical Soc., 1990.

\bibitem{Le}
 H.~Lewy,
{\em
An example of a smooth linear partial differential equation without solution},  Ann. of Math. (2) {\bf 66.1} (1957), 155-158.



\bibitem{Po} H.  Poincar{\'e}, {\em Les fonctions analytiques de deux variables et la repr{\'e}sentation conforme},
    Rendiconti del Circolo Matematico di Palermo {\bf  23.1} (1884-1940), 185--220.


\bibitem{Ro} H.~Rossi, 
  {\em Attaching analytic spaces to an analytic space along a pseudoconcave boundary}. In: 
Proceedings of the conference on complex analysis (1965), 242--256. Springer.

\bibitem{Ta} N.~Tanaka,
  {\em On the pseudo-conformal geometry of hypersurfaces of the space of n complex variables},
 J. Math. Soc. Japan
  {\bf  14.4} (1962), 397--429. 
  
    \bibitem{Tr}F.~Treve, {\em On local solvability of linear partial differential equations},
Bull. Amer. Math. Soc. {\bf 76}  (1970), 552-571.  

\bibitem{TrBook}
F.~Treves, 
{\em Hypo-analytic structures: Local theory},
 Princeton Mathematical Series, vol. 40. Princeton University Press, Princeton, NJ, 1992.


 
\bibitem{We} S.M.~Webster, 
  {\em Pseudo-Hermitian structures on a real hypersurface},
   J. Diff. Geom.
  {\bf  13.1}  (1978), 25--41.

 

\bibitem{Weyl}
 H. Weyl,
{\em
Cartan on Groups and Differential Geometry}, 
Bull. Amer. Math. Soc. {\bf  44.9}, Part 1 (1938), 598-607.

\end{thebibliography}
\end{document}